\title{Building spanning trees quickly in Maker-Breaker games}
\author{ 
{Dennis Clemens \thanks{Department of Mathematics and Computer Science, Freie Universit\"{a}t Berlin, Germany. Email: d.clemens@fu-berlin.de. Research supported by DFG, project SZ 261/1-1.}} 
\quad{Asaf Ferber \thanks{School of Mathematical Sciences, Raymond and Beverly Sackler Faculty of Exact Sciences, Tel Aviv University, Tel Aviv, 69978, Israel. Email: ferberas@post.tau.ac.il.}}
\quad{Roman Glebov \thanks{Mathematics Institute and DIMAP, University of Warwick, Coventry CV4 7AL, UK. Previous affiliation: Institut f\"ur Mathematik, Freie Universit\"{a}t Berlin, Arnimallee 3-5, D-14195 Berlin, Germany. Email: glebov@zedat.fu-berlin.de. Research supported by DFG within the research training group "Methods for Discrete Structures".}}\\
\quad {Dan Hefetz \thanks{School of Mathematics,
University of Birmingham, Edgbaston, Birmingham B15 2TT, United
Kingdom. Email: d.hefetz@bham.ac.uk. Research supported by an EPSRC Institutional Sponsorship Fund.}}  
\quad{Anita Liebenau \thanks{Department of Mathematics and Computer Science, Freie Universit\"{a}t Berlin, Germany. Email: liebenau@math.fu-berlin.de. Research supported by DFG within the graduate school Berlin Mathematical School.}}}

\documentclass[11pt]{article}
\usepackage{amsmath,amssymb,latexsym,color,epsfig,a4,enumerate}

\newif\ifnotesw\noteswtrue

\newtheorem{theorem}{Theorem}[section]
\newtheorem{lemma}[theorem]{Lemma}
\newtheorem{claim}[theorem]{Claim}

\newtheorem{conjecture}[theorem]{Conjecture}

\newtheorem{remark}[theorem]{Remark}

\let\theta=\vartheta
\let\rho=\varrho
\let\sigma=\varsigma

\def\D{\Delta}

\newenvironment{proof}{\noindent{\bf Proof\,}}{\hfill$\Box$}

\newcommand{\open}{\ensuremath{\mathcal{O}}}

\begin{document}
\maketitle

\begin{abstract}
For a tree $T$ on $n$ vertices, we study the Maker-Breaker game, played on the edge set of the
complete graph on $n$ vertices, which Maker wins as soon as the graph she builds contains a copy 
of $T$. We prove that if $T$ has bounded maximum degree, then Maker can win this game within
$n+1$ moves. Moreover, we prove that Maker can build almost every tree on $n$ vertices in $n-1$
moves and provide non-trivial examples of families of trees which Maker cannot build in $n-1$ moves.    
\end{abstract}

\section{Introduction}
Let $X$ be a finite set and let ${\mathcal F} \subseteq 2^X$ be a family of subsets. In the Maker-Breaker 
game $(X, {\mathcal F})$, two players, called Maker and Breaker, take turns in claiming a previously unclaimed
element of $X$, with Breaker going first. The set $X$ is called the board of the game and the members of 
${\mathcal F}$ are referred to as the winning sets. Maker wins this game as soon as she claims all elements 
of some winning set. If Maker does not fully claim any winning set by the time every board element is claimed 
by some player, then Breaker wins the game. We say that the game $(X, {\mathcal F})$ is Maker's win if
Maker has a strategy that ensures her win in this game (in some number of moves) against any strategy of Breaker, 
otherwise the game is Breaker's win. One can also consider a \emph{biased} version in which Maker claims $p$
board elements per move (instead of just 1) and Breaker claims $q$ board elements per move. We refer to this
version as a $(p : q)$ game. For a more detailed discussion, we refer the reader to~\cite{BeckBook}.

The following game was studied in~\cite{FHK2012}. Let $T$ be a tree on $n$ vertices. The board of the
\emph{tree embedding game} $(E(K_n), {\mathcal T}_n)$ is the edge set of the complete graph on $n$ vertices 
and the minimal (with respect to inclusion) winning sets are the labeled copies of $T$ in $K_n$. Several
variants of this game were studied by various researchers (see e.g.~\cite{Beck94, Bednarska, JKS}).  

It was proved in~\cite{FHK2012} that for any real numbers $0 < \alpha < 0.005$ and $0 < \varepsilon < 0.05$ 
and a sufficiently large integer $n$, Maker has a strategy to win the $(1 : q)$ game $(E(K_n), {\mathcal T}_n)$ 
within $n + o(n)$ moves, for every $q \leq n^{\alpha}$ and every tree $T$ with $n$ vertices and maximum degree
at most $n^{\varepsilon}$. The bounds on the duration of the game, on Breaker's bias and on the maximum degree
of the tree to be embedded, do not seem to be best possible. Indeed, it was noted in~\cite{FHK2012} that
it would be interesting to improve each of these bounds, even at the expense of the other two. In this paper 
we focus on the duration of the game. We restrict our attention to the case of bounded degree trees 
and to unbiased games (that is, the case $q=1$). 

The smallest number of moves Maker needs in order to win some Maker-Breaker game is an important game invariant 
which has received a lot of attention in recent years (see e.g.~\cite{Beck, CFKL2012, FH2011, FH, FHK2012, Gebauer, 
HKSS, HSt, Pekec}). Part of the interest in this invariant stems from its usefulness in the study of strong games.
In the \emph{strong game} $(X, {\mathcal F})$, two players, called Red and Blue, take turns in claiming one previously 
unclaimed element of $X$, with Red going first. The winner of the game is the \emph{first} player to fully claim some 
$F \in {\mathcal F}$. If neither player is able to fully claim some $F \in {\mathcal F}$ by the time every element of 
$X$ has been claimed by some player, the game ends in a \emph{draw}. Strong games are notoriously hard to analyze. 
For certain strong games, a combination of a \emph{strategy stealing} argument and a \emph{hypergraph coloring} 
argument can be used to prove that these games are won by Red. However, the aforementioned arguments are purely 
existential. That is, even if it is known that Red has a winning strategy for some strong game $(X, {\mathcal F})$, 
it might be very hard to describe such a strategy explicitly. The use of explicit very fast winning strategies for 
Maker in a weak game for devising an explicit winning strategy for Red in the corresponding strong game was initiated 
in~\cite{FH2011}. This idea was used to devise such strategies for the strong perfect matching and Hamilton 
cycle games~\cite{FH2011} and for the $k$-vertex-connectivity game~\cite{FH}.    

Returning to the tree embedding game $(E(K_n), {\mathcal T}_n)$, it is obvious that Maker cannot build any tree on 
$n$ vertices in less than $n-1$ moves. This trivial lower bound can be attained for some trees. For example, it was 
proved in~\cite{HKSS} that Maker can build a Hamilton path of $K_n$ in $n-1$ moves. On the other hand it is not hard 
to see that there are trees on $n$ vertices which Maker cannot build in less than $n$ moves. Indeed, consider
for example the complete binary tree on $n$ vertices $BT_n$. Suppose for a contradiction that Maker
can build a copy of $BT_n$ in $n-1$ moves. It follows that after $n-2$ moves, Maker's graph is isomorphic to
$BT_n \setminus e$, where $e$ is some edge of $BT_n$. Note that for any $e \in E(BT_n)$, there is a unique edge 
of $K_n$ which Maker has to claim in order to complete a copy of $BT_n$. Hence, by claiming this edge, Breaker 
delays Maker's win by at least one move. Note that, in contrast, if $e$ is an edge of a path $P_n$ which is not
incident with any of its endpoints, then there are four edges of $K_n$ whose addition to a copy of $P_n \setminus e$
yields a copy of $P_n$.         

In this paper we prove the following general upper bound which is only one move away from
the aforementioned lower bound.

\begin{theorem} \label{WeakGame}
Let $\Delta$ be a positive integer. Then there exists an integer
$n_0 = n_0(\Delta)$ such that for every $n \geq n_0$ and for every
tree $T=(V,E)$ with $|V|=n$ and $\Delta(T) \leq \Delta$, Maker has a strategy 
to win the game $(E(K_n), {\mathcal T}_n)$ within $n+1$ moves.
\end{theorem}

Since every tree either does or does not admit a long bare path (a path of a tree $T$ is called \emph{bare} if all 
of its interior vertices are of degree 2 in $T$) we will deduce Theorem~\ref{WeakGame} 
as an immediate corollary of the following two theorems (with $m_2 = m_1$ and $n_0 = \max \{n_1, n_2\}$). 

\begin{theorem} \label{th::longBarePath}
Let $\Delta$ be a positive integer. Then there exists an integer $m_1 = m_1(\Delta)$ and an
integer $n_1 = n_1(\Delta, m_1)$ such that the following holds for every $n \geq n_1$ and for every
tree $T=(V,E)$ with $|V|=n$ and $\Delta(T) \leq \Delta$. If $T$ admits a bare path
of length $m_1$, then Maker has a strategy to win the game $(E(K_n), {\mathcal T}_n)$ within $n$ moves.
\end{theorem}

\begin{theorem} \label{th::manyLeaves}
Let $\Delta$ and $m_2$ be positive integers. Then there exists an integer
$n_2 = n_2(\Delta, m_2)$ such that the following holds for every $n \geq n_2$ and for every
tree $T=(V,E)$ such that $|V|=n$ and $\Delta(T) \leq \Delta$. If $T$ does not admit a bare path
of length $m_2$, then Maker has a strategy to win the game $(E(K_n), {\mathcal T}_n)$ within $n+1$ moves.
\end{theorem}

Recall that Maker cannot build a copy of the complete binary tree on $n$ vertices in less than $n$ moves. 
One can adapt the argument used to prove this statement to obtain many examples of trees which Maker cannot 
build in $n-1$ moves. Nevertheless, the following theorem suggests that such examples are quite rare.

\begin{theorem} \label{th::randomTree}
Let $T$ be a tree, chosen uniformly at random from the class of all labeled trees on $n$ vertices. Then 
asymptotically almost surely, $T$ is such that Maker has a strategy to win the game $(E(K_n), {\mathcal T}_n)$
in $n-1$ moves.
\end{theorem}

One of the main ingredients in our proof of Theorem~\ref{th::randomTree} is the construction of a Hamilton path 
with one designated endpoint in optimal time (see Lemma~\ref{lem::PathOneEndpoint}). Using this lemma it will be 
easy to obtain the following generalization of Theorem 1.4 from~\cite{HKSS}.

\begin{theorem} \label{th::PathFromLeaf}
Let $\Delta$ be a positive integer. Then there exists an integer $m_3 = m_3(\Delta)$ and an
integer $n_3 = n_3(\Delta, m_3)$ such that the following holds for every $n \geq n_3$ and for every
tree $T=(V,E)$ with $|V|=n$ and $\Delta(T) \leq \Delta$. If $T$ admits a bare path of length $m_3$, 
such that one of its endpoints is a leaf of $T$, then Maker has a strategy to win the game 
$(E(K_n), {\mathcal T}_n)$ in $n-1$ moves.
\end{theorem}

The rest of this paper is organized as follows: in Subsection~\ref{sec::prelim} we introduce some notation
and terminology that will be used throughout this paper. In Section~\ref{sec::TwoTrees} we prove 
Theorem~\ref{th::longBarePath}, in Section~\ref{sec:manyleaves} we prove Theorem~\ref{th::manyLeaves} and
in Section~\ref{sec::RandomTree} we prove Theorems~\ref{th::randomTree} and~\ref{th::PathFromLeaf}. 
Finally, in Section~\ref{sec::openprob} we present some open problems.

\subsection{Notation and terminology} \label{sec::prelim}
\noindent Our graph-theoretic notation is standard and follows that of~\cite{West}. 
In particular, we use the following.

For a graph $G$, let $V(G)$ and $E(G)$ denote its sets of vertices
and edges respectively, and let $v(G) = |V(G)|$ and $e(G) = |E(G)|$.
For disjoint sets $A,B \subseteq V(G)$, let $E_G(A,B)$ denote the
set of edges of $G$ with one endpoint in $A$ and one endpoint in
$B$, and let $e_G(A,B) = |E_G(A,B)|$. For a set $S \subseteq V(G)$,
let $G[S]$ denote the subgraph of $G$ which is induced on the set
$S$. For disjoint sets $S,T \subseteq V(G)$, let $N_G(S,T) = \{u 
\in T : \exists v \in S, uv \in E(G)\}$ denote the set of
neighbors of the vertices of $S$ in $T$.
For a set $T \subseteq V(G)$ and a vertex $w \in V(G) \setminus T$ 
we abbreviate $N_G(\{w\}, T)$ to $N_G(w, T)$, and let $d_G(w,T)
= |N_G(w,T)|$ denote the degree of $w$ into $T$. For a set $S \subseteq V(G)$ 
and a vertex $w \in V(G)$ we abbreviate $N_G(S, V(G) \setminus S)$ to $N_G(S)$
and $N_G(w, V(G) \setminus \{w\})$ to $N_G(w)$. We let $d_G(w)
= |N_G(w)|$ denote the degree of $w$ in $G$. The minimum and maximum
degrees of a graph $G$ are denoted by $\delta(G)$ and $\Delta(G)$
respectively. Often, when there is no 
risk of confusion, we omit the subscript $G$ from the notation above.
Let $P = (v_1 v_2 \ldots v_k)$ be a path in a graph $G$. The vertices $v_1$ 
and $v_k$ are called the \emph{endpoints} of $P$, whereas the vertices of 
$V(P) \setminus \{v_1, v_k\}$ are called the \emph{interior vertices} of $P$. 
We denote the set of endpoints of a path $P$ by $End(P)$. Note that $|End(P)| 
= \min \{2, v(P)\}$. The \emph{length} of a path is the number of its edges. 
A path of a tree $T$ is called a \emph{bare path} if all 
of its interior vertices are of degree 2 in $T$. Given two graphs $G$ and $H$ 
on the same set of vertices $V$, let $G \setminus H$ denote the graph with 
vertex set $V$ and edge set $E(G) \setminus E(H)$.

Let $G$ be a graph, let $T$ be a tree, and let $S \subseteq V(T)$ be
an arbitrary set. An \emph{$S$-partial embedding} of $T$ in $G$ is
an injective mapping $f : S \rightarrow V(G)$, such that $f(x)f(y)
\in E(G)$ whenever $x,y \in S$ and $xy \in E(T)$. For a vertex 
$v \in f(S)$ let $v' = f^{-1}(v)$ denote its pre-image under $f$. 
If $S = V(T)$, we call an $S$-partial embedding of $T$ in $G$ simply an embedding of
$T$ in $G$. We say that the vertices of $S$ are \emph{embedded},
whereas the vertices of $V(T) \setminus S$ are called \emph{new}. An
embedded vertex is called \emph{closed} with respect to $T$ and $f$ if all
its neighbors in $T$ are embedded as well. An embedded vertex, that
is not closed with respect to $T$ and $f$, is called \emph{open} with
respect to $T$ and $f$. The vertices of $f(S)$ are called \emph{taken},
whereas the vertices of $V(G) \setminus f(S)$ are called
\emph{available}. With some abuse of this terminology, for a closed
(respectively open) vertex $u' \in S$, we sometimes refer to
$f(u')$ as being closed (respectively open) as well. Moreover,
we omit the phrase ``with respect to $T$ and $f$'' 
or abbreviate it to ``with respect to $T$'', if its meaning is 
clear from the context. In particular we denote the set of open 
vertices with respect to $T$ and $f$ by ${\mathcal O}_T$. 

Assume that some Maker-Breaker game, played on the edge set of some
graph $G$, is in progress. At any given moment during this game, we
denote the graph spanned by Maker's edges by $M$ and the graph
spanned by Breaker's edges by $B$; the edges of $G \setminus (M \cup B)$ 
are called \emph{free}.

\section{Trees which admit a long bare path} \label{sec::TwoTrees}

In this section we will prove Theorem~\ref{th::longBarePath}. The main idea is to first embed 
the tree $T$ except for a sufficiently long bare path $P$ and then to embed $P$ between its previously 
embedded endpoints. In the first stage we will waste no moves, whereas in the second we will waste at most 
one. Starting with the former we prove the following result. 

\begin{theorem} \label{TreesWithLongPath}
Let $r$ be a positive integer and let $n, m$ and $\D \geq 3$ be integers satisfying $n > m \geq (\D+r)^2$. 
For every $1 \leq i \leq r$, let $T_i = (V_i, E_i)$ be a tree with maximum degree at most $\D$ and
assume that $\sum_{i=1}^r |V_i| = n - m$. For every $1 \leq i \leq r$ let $x'_i \in V_i$ be an 
arbitrary vertex. Then, playing a Maker-Breaker game 
on the edge set of $K_n$, Maker has a strategy to ensure that the following two properties will hold 
immediately after her $(\sum_{i=1}^r |V_i| - r)$th move: 
\begin{description}
\item [(i)] $M \cong \bigcup_{i=1}^r T_i$, that is, Maker's graph is a vertex disjoint union of the $T_i$'s. 
\item [(ii)] There exists an isomorphism $f : \bigcup_{i=1}^r T_i \rightarrow M$ for which $e_B(A \cup f(\{x'_1, \ldots, x'_r\})) \leq \binom{\D+r-1}{2}$, where $A = V(K_n) \setminus f(\bigcup_{i=1}^r V_i)$ is the set of available vertices. 
\end{description}
\end{theorem}

\begin{remark} \label{rem::generalr}
In the proof of Theorem~\ref{th::longBarePath} we will use the special case $r=2$ of Theorem~\ref{TreesWithLongPath}. Another special case, namely $r=1$, will be used in the proof of Theorem~\ref{th::PathFromLeaf}. It is therefore convenient to prove it here for every $r$. Moreover, it might have future applications where other values of $r$ are considered.   
\end{remark}

\textbf{Proof of Theorem~\ref{TreesWithLongPath}}
We begin by describing Maker's strategy. At any point during the game, if Maker is unable to follow the proposed strategy, 
then she forfeits the game. We will prove that Maker can follow this strategy without forfeiting the game and that, 
by doing so, she wins the game.

\textbf{Maker's strategy:} Throughout the game, Maker maintains a set $S \subseteq \bigcup_{i=1}^r V_i$ 
of embedded vertices, an $S$-partial embedding $f$ of $\bigcup_{i=1}^r T_i$ in $K_n \setminus B$ and
a set $A = V(K_n) \setminus f(S)$ such that $e_B(A \cup f(\{x'_1, \ldots, x'_r\})) \leq \binom{\D+r-1}{2}$. 
Initially $S = \{x'_1, \ldots, x'_r\}$, $f(x'_i) = x_i$ for every $1 \leq i \leq r$ where 
$x_1, \ldots, x_r \in V(K_n)$ are $r$ arbitrary vertices, and $A = V(K_n) \setminus \{x_1, \ldots, x_r\}$. 
At any point during the game we denote the set $A \cup \{x_1, \ldots, x_r\}$ by $U$.

Maker's strategy is based on the following potential function: for every
vertex $u \in V(K_n)$ let $\phi(u) = \max \{0, d_B(u,U) - d_M(u)\}$ 
and let 
$$
\psi = e_B(U) + \sum_{i=1}^r \sum_{w \in f({\mathcal O}_{T_i})} \phi(w)
$$ 
(by abuse of notation we use $\psi$ to denote the potential at any point during the game).
 
For every $1 \leq i \leq r$ let $d_i = d_{T_i}(x'_i)$. In her first $\sum_{i=1}^r d_i$ moves, 
Maker closes $x'_1, \ldots, x'_r$, that is, for every $1 \leq i \leq r$ and every $1 \leq j \leq d_i$ 
she claims a free edge $x_i y_{ij}$ where the elements of $\{y_{ij} : 1 \leq i \leq r, 1 \leq j \leq d_i\}$ 
are $\sum_{i=1}^r d_i$ arbitrary vertices of $A$. She then updates $A, U, S$ and $f$ as follows. 
For every $1 \leq i \leq r$ let $y'_{i1}, \ldots, y'_{id_i}$ be the neighbours of $x'_i$ in $T_i$. 
Maker deletes the elements of $\{y_{ij} : 1 \leq i \leq r, 1 \leq j \leq d_i\}$ from $A$ (and then 
also from $U$), adds the elements of $\{y'_{ij} : 1 \leq i \leq r, 1 \leq j \leq d_i\}$ to $S$
and sets $f(y'_{ij}) = y_{ij}$ for every $1 \leq i \leq r$ and every $1 \leq j \leq d_i$.  

For every integer $\ell > \sum_{i=1}^r d_i$, Maker plays her $\ell$th move according 
to the value of $\psi$ at that time. She distinguishes between the following three cases.

\begin{description} 
\item [Case 1: $\psi \leq \binom{\D+r-1}{2}$.] 
Maker claims a free edge $vz$ such that the following properties hold:
\begin{enumerate}
\item $v \in \bigcup_{i=1}^r f({\mathcal O}_{T_i})$;
\item $z \in A$. 
\end{enumerate}
Subsequently, Maker updates $A, U, S$ and $f$ by deleting $z$ from $A$ (and then also from $U$), adding $z'$ to $S$
and setting $f(z') = z$, where $z'$ is an arbitrary new neighbor of $f^{-1}(v)$ in $\bigcup_{i=1}^r T_i$.

\item [Case 2: $\psi > \max \left\{\binom{\D+r-1}{2}, e_B(U) \right\}$.] 
Maker claims a free edge $vz$ such that the following properties hold:
\begin{enumerate}
\item $v \in \bigcup_{i=1}^r f({\mathcal O}_{T_i})$;
\item $d_B(v,U) > d_M(v)$;
\item $z \in A$. 
\end{enumerate}
Subsequently, Maker updates $A, U, S$ and $f$ as in Case 1.

\item [Case 3: $\psi = e_B(U) > \binom{\D+r-1}{2}$.] 
Maker claims a free edge $vz$ such that the following properties hold:
\begin{enumerate}
\item $v \in \bigcup_{i=1}^r f({\mathcal O}_{T_i})$;
\item $z \in A$. 
\item $d_B(z,U) > 0$.
\end{enumerate}
Subsequently, Maker updates $A, U, S$ and $f$ as in Case 1.
\end{description}

We wish to prove that Maker can follow the proposed strategy without forfeiting the game. Note first that 
$\psi \geq e_B(U)$ holds by definition and thus Maker will never face a situation which is not covered by   
Cases 1,2 and 3 above. Next, we prove the following claims.

\begin{claim} \label{cl::Maker}
For every $\sum_{i=1}^r d_i < \ell \leq \sum_{i=1}^r |V_i| - r$, Maker does not increase $\psi$
in her $\ell$th move.
\end{claim}

\begin{proof}
For every $\sum_{i=1}^r d_i < \ell \leq \sum_{i=1}^r |V_i| - r$, in her $\ell$th move Maker
claims an edge $vz$ such that $v \in \bigcup_{i=1}^r f({\mathcal O}_{T_i})$ and $z \in A$. Clearly,
this does not affect $\phi(u)$ for any $u \in V(K_n) \setminus \{v,z\}$. Moreover, $\phi(v)$ is not
increased, $e_B(U)$ is decreased by $d_B(z,U)$ and $\sum_{i=1}^r \sum_{w \in f({\mathcal O}_{T_i})} \phi(w)$ 
is increased by at most $\phi(z) \leq d_B(z,U)$.   
\end{proof}

\begin{claim} \label{obs_p}
$\psi \leq \binom{\D+r-1}{2}$ holds immediately after Maker's $\ell$th move for every 
$\sum_{i=1}^r d_i \leq \ell \leq \sum_{i=1}^r |V_i| - r$.
\end{claim}

\begin{proof}
We prove this by induction on the number of Maker's moves. Since $\left(\bigcup_{i=1}^r f({\mathcal O}_{T_i})\right) \cap \{x_1, \ldots, x_r\} = \emptyset$ holds after Maker's $(\sum_{i=1}^r d_i)$th move, it follows that, from this point onwards, every edge $e \in E(B)$ contributes at most $1$ to $\psi$. Since $\D \geq 3$ it thus follows that $\psi \leq \sum_{i=1}^r d_i \leq r\D \leq \binom{\D+r-1}{2}$ holds immediately after Maker's $(\sum_{i=1}^r d_i)$th move. Assume that $\psi \leq \binom{\D+r-1}{2}$ holds 
immediately after her $\ell$th move for some $\sum_{i=1}^r d_i \leq \ell < \sum_{i=1}^r |V_i| - r$; we will 
show that, unless Maker forfeits the game, this inequality holds immediately after her $(\ell+1)$st move as well. 
Since, $x'_1, \ldots, x'_r$ are closed, from now on Breaker can increase $\psi$ by at most $1$ per move. 
It thus follows by the induction hypothesis that $\psi \leq \binom{\D+r-1}{2} + 1$ holds immediately before Maker's
$(\ell+1)$st move. Assume first that in fact $\psi \leq \binom{\D+r-1}{2}$. It follows by Claim~\ref{cl::Maker} that
$\psi \leq \binom{\D+r-1}{2}$ holds immediately after Maker's $(\ell+1)$st move as well. Assume then that 
$\psi = \binom{\D+r-1}{2} + 1$; it suffices to prove that Maker decreases $\psi$ by at least $1$ in her $(\ell+1)$st move. 
Maker plays according to the proposed strategy, either for Case 2 or for Case 3. In Case 2 $\psi$ is not increased 
since the value of $\sum_{i=1}^r \sum_{w \in f({\mathcal O}_{T_i})} \phi(w)$ is increased by at most $d_B(z,U)$ 
and the value of $e_B(U)$ is decreased by the same amount. Moreover, since $d_B(v,U) > d_M(v)$, it follows that $\phi(v)$ 
is decreased by at least $1$. Since $v \in \bigcup_{i=1}^r f({\mathcal O}_{T_i})$ holds before Maker's $(\ell+1)$st move, 
we conclude that $\psi$ is decreased by at least $1$. In Case 3 Maker decreases $e_B(U)$ by $d_B(z,U)$. Moreover, if $z$ 
becomes closed, then $\sum_{i=1}^r \sum_{w \in f({\mathcal O}_{T_i})} \phi(w)$ is not increased, whereas, if $z$ 
becomes open, then since $d_B(z,U) > 0$, it is increased by $d_B(z,U) - d_M(z) = d_B(z,U) - 1$. 
Either way, $\psi$ is decreased by at least $1$.      
\end{proof}

We can now prove that Maker is indeed able to play according to the proposed strategy.

\begin{claim} \label{claim_p} 
Maker can follow the proposed strategy without forfeiting the game for $\sum_{i=1}^r |V_i| - r$ moves.
\end{claim}

\begin{proof}
Since Maker aims to build a copy of $\bigcup_{i=1}^r T_i$ within $\sum_{i=1}^r |V_i| - r$ moves and since 
$\sum_{i=1}^r |V_i| = n - m \leq n - (\D+r)^2$, it follows that $|A| \geq (\D+r)^2$ holds at any point during these
$\sum_{i=1}^r |V_i| - r$ moves; in particular Maker can follow the first $\sum_{i=1}^r d_i$ moves of the proposed strategy. 
As previously noted, once $x'_1, \ldots, x'_r$ are closed, 
Breaker can increase $\psi$ by at most $1$ per move. It thus follows by Claim~\ref{obs_p} that 
$\psi \leq \binom{\D+r-1}{2} + 1$ holds at any point during the remainder of the game. Assume first that 
$\psi \leq \binom{\D+r-1}{2}$. Let $v \in \bigcup_{i=1}^r f({\mathcal O}_{T_i})$, then 
$\phi(v) \leq \psi \leq \binom{\D+r-1}{2}$ and thus $d_B(v,U) \leq \phi(v) + d_M(v) \leq \binom{\D+r-1}{2}
+ \D < (\D + r)^2 \leq |A|$. Hence there exists a free edge $vz$ such that $z \in A$. We conclude that Maker 
can follow her strategy for Case 1.

Assume then that $\psi = \binom{\D+r-1}{2} + 1$. Assume further that $\psi > e_B(U)$. It follows
that there exists a vertex $v \in \bigcup_{i=1}^r f({\mathcal O}_{T_i})$ 
such that $\phi(v) > 0$ and thus $d_B(v,U) > d_M(v)$. The same calculation as above shows that
$d_B(v,U) < |A|$. Therefore, Maker can claim a free edge $vz$ as required by her strategy for Case 2.

Assume then that $e_B(U) = \psi = \binom{\D+r-1}{2} + 1$. It follows that there are at least $\D+r$ vertices 
$z \in U$ for which $d_B(z,U) > 0$; by the definition of $A$ and $U$, at least $\D$ of them must be in $A$. Let $v \in \bigcup_{i=1}^r f({\mathcal O}_{T_i})$. Since $\psi = e_B(U)$, it follows that $\phi(v) = 0$ and thus $d_B(v,U) \leq d_M(v) < \D$ (the last inequality 
holds since $v$ is open). Therefore, Maker can claim a free edge $vz$ as required by her strategy for Case 3. 
\end{proof}

Since Maker follows the proposed strategy, it is evident that after $\sum_{i=1}^r |V_i| - r$ 
moves she builds a graph which is isomorphic to $\bigcup_{i=1}^r T_i$. Moreover, since $\phi(w) \geq 0$ 
for every vertex $w$, it follows by Claim~\ref{obs_p} that $e_B(U) \leq \psi \leq \binom{\D+r-1}{2}$
holds, in particular, immediately after Maker's $(\sum_{i=1}^r |V_i| - r)$th move. We conclude
that Maker can indeed ensure that Properties (i) and (ii) will hold immediately after her
$(\sum_{i=1}^r |V_i| - r)$th move. 
{\hfill $\Box$ \medskip\\}

\bigskip

Next, we wish to embed a Hamilton path whose endpoints were previously embedded, into
an almost complete graph. Formally, we need the following result. 

\begin{lemma}\label{xy-path}
For every positive integer $k$ there exists an integer $m_0 = m_0(k)$
such that the following holds for every $m \geq m_0$. Let $G$ be a graph
with $m$ vertices and $e(G) \geq \binom{m}{2} - k$ edges and let $x$ and $y$ be
two arbitrary vertices of $G$. Then, playing a Maker-Breaker game on $E(G)$, Maker
has a strategy to build a Hamilton path of $G$ between
$x$ and $y$ within $m$ moves. 
\end{lemma}

Lemma~\ref{xy-path} can be proved similarly to Theorem 1.1 from~\cite{HSt}. We omit the straightforward details.

We can now combine Theorem~\ref{TreesWithLongPath} and Lemma~\ref{xy-path} to deduce Theorem~\ref{th::longBarePath}. 

\textbf{Proof of Theorem~\ref{th::longBarePath}}
Let $k = \binom{\D + 1}{2} + 1$, let $m_0 = m_0(k)$ be the constant whose existence follows
from Lemma~\ref{xy-path} and let $m_1 = \max \{m_0, (\D+2)^2\}$. Let $P$ be a bare path in $T$ 
of length $m_1$ with endpoints $x'_1$ and $x'_2$. Let $F$ be the forest which is obtained from $T$ 
by deleting all the vertices in $V(P) \setminus \{x'_1, x'_2\}$. Let $T_1$ be the connected component
of $F$ which contains $x'_1$ and let $T_2$ be the connected component of $F$ which contains $x'_2$.

Maker's strategy consists of two stages. In the first stage she embeds $T_1 \cup T_2$ using the strategy
whose existence follows from Theorem~\ref{TreesWithLongPath} (with $r=2$) while ensuring that Properties 
(i) and (ii) are satisfied. Let $f : T_1 \cup T_2 \rightarrow M$ be an isomorphism, let $x_1 = f(x'_1)$, 
let $x_2 = f(x'_2)$, let $A = V(K_n) \setminus f(V(T_1) \cup V(T_2))$, let $U = A \cup \{x_1, x_2\}$ and 
let $G = (K_n \setminus B)[U]$.

In the second stage she embeds $P$ into $G$ between the endpoints $x_1$ and $x_2$. She does so 
using the strategy whose existence follows from Lemma~\ref{xy-path} which is applicable by the choice 
of $m_1$ and by Property (ii). Hence, $T \subseteq M$ holds at the end of the second stage, that is, 
Maker wins the game.

It follows by Theorem~\ref{TreesWithLongPath} that the first stage lasts exactly 
$v(T_1) + v(T_2) - 2 = n - |V(P)| = n - |U|$ moves. It follows by Lemma~\ref{xy-path} that the second 
stage lasts at most $|U|$ moves. Therefore, the entire game lasts at most $n$ moves as claimed. 
{\hfill $\Box$ \medskip\\}

\section{Trees which do not admit a long bare path} \label{sec:manyleaves}
In this section we will prove Theorem~\ref{th::manyLeaves}. The main idea is to first embed 
the tree $T$ except for a large matching between some of its leaves and their parents and then 
to embed this matching between the previously embedded endpoints and the remaining available
vertices. In the first stage we will waste no moves, whereas in the second we will waste at most 
two. 

In order for this approach to be valid, we must first prove that such a matching exists
in $T$.

\begin{lemma} \label{lem::NoPathLargeMatching}
For all positive integers $\Delta$ and $m$ there exists an integer $n_0 = n_0(\Delta,m)$
such that the following holds for every $n \geq n_0$. Let $T$ be a tree on $n$ vertices with
maximum degree at most $\Delta$ and let $L$ denote the set of leaves of $T$. If $T$ does not 
admit a bare path of length $m$, then $|L| \geq |N_T(L)| \geq \frac{n}{2 \Delta (m+1)}$.
\end{lemma}  

The inequality $|L| \geq |N_T(L)|$ is trivial. Moreover, since the maximum degree 
of $T$ is at most $\Delta$, it follows that $|L| \leq \Delta \cdot |N_T(L)|$. Hence, 
Lemma~\ref{lem::NoPathLargeMatching} is an immediate corollary of the following result 
(with $k = m$ and $\ell = |L|$).

\begin{lemma} [Lemma 2.1 in~\cite{Kri}] \label{MichaelTreeDistinction}
Let $k, n$ and $\ell$ be positive integers. Let $T$ be a tree on $n$ vertices with at most $\ell$ leaves. 
Then $T$ contains a collection of at least $\frac{n - (2 \ell - 2)(k+1)}{k + 1}$
vertex disjoint bare paths of length $k$ each.
\end{lemma}

Next, we prove that Maker can build a perfect matching very quickly when playing on the 
edge set of a very dense subgraph of a sufficiently large complete bipartite graph.

Let $G = (V,E)$ be a graph. The winning sets of the \emph{perfect matching game}, played
on the board $E$, are the edge sets of all matchings of $G$ of size $\lfloor |V|/2 \rfloor$.
The following theorem was proved in~\cite{HKSS}.

\begin{theorem} [Theorem 1.2 in~\cite{HKSS}] \label{PMKn}
There exists an integer $n_0$ such that for every $n \geq n_0$, Maker has a strategy to 
win the perfect matching game, played on $E(K_n)$, within $\lfloor n/2 \rfloor + (n+1) \mod 2$ 
moves.
\end{theorem}

The following analogous result, which applies to the perfect matching game, played on 
a complete bipartite graph, holds as well.

\begin{theorem} \label{th::PM}
There exists an integer $n_0$ such that for every $n \geq n_0$, Maker has a strategy to 
win the perfect matching game, played on $E(K_{n,n})$, within $n+1$ moves.
\end{theorem}

One can prove Theorem~\ref{th::PM} using essentially the same argument as in the proof 
of Theorem~\ref{PMKn} given in~\cite{HKSS}. We omit the straightforward details and refer 
the reader to~\cite{HKSS}.

The following lemma, which will be used in the proof of Theorem~\ref{th::manyLeaves}, 
asserts that Maker can win the perfect matching game very quickly even when the board 
is a very dense subgraph of a sufficiently large complete bipartite graph.

\begin{lemma} \label{FastPerMatBP}
For all non-negative integers $k_1$ and $k_2$ there exists an integer $f(k_1,k_2)$ such
that the following holds for every $n \geq f(k_1,k_2)$. Let $G = (U_1 \cup U_2, E)$ 
be a bipartite graph which satisfies the following properties:
\begin{description}
\item [(i)] $|U_1| = |U_2| = n$;
\item [(ii)] $d(u_1, U_2) \geq n - k_1$ for every $u_1 \in U_1$;
\item [(iii)] $d(u_2, U_1) \geq n - k_2$ for every $u_2 \in U_2$.
\end{description}
Then Maker has a strategy to win the perfect matching game, played on $E$, 
within $n+2$ moves.
\end{lemma}

\begin{remark} \label{rem::optimal}
The bound on the number of moves given in Lemma~\ref{FastPerMatBP} is best possible, even
for the case $k_1 = k_2 = 1$. Indeed, one can show that, when playing on $K_{n,n}$ from 
which a perfect matching was removed, Maker cannot build a perfect matching within $n+1$
moves; we omit the details. 
\end{remark}

\textbf{Proof of Lemma~\ref{FastPerMatBP}}
The following notation and terminology will be used throughout this proof. 
At any point during the game, let $S$ denote the set of vertices of $G$ 
which are isolated in Maker's graph, let $S_1 = S \cap U_1$ and let 
$S_2 = S \cap U_2$. Let $Br = ((K_{n,n} \setminus G) \cup B)[S]$. For $i \in \{1,2\}$ 
let $\Delta_i = \max \{d_{Br}(w) : \, w \in S_i\}$.

We prove Lemma~\ref{FastPerMatBP} by induction on $k_1 + k_2$. 
In the induction step we will need to assume that $k_1 + k_2 \geq 3$. 
Hence, we first consider the case $k_1 + k_2 \leq 2$. Note that if
$k_1 = 0$, then $k_2 = 0$ and vice versa. Since, moreover, the case 
$k_1 = k_2 = 0$ follows directly from Theorem~\ref{th::PM}, it suffices 
to consider the case $k_1 = k_2 = 1$. In this case $K_{n,n} \setminus G$
is a matching. Let $U_1 = \{x_1, \ldots, x_n\}$ and $U_2 = \{y_1, \ldots, y_n\}$ and 
assume without loss of generality that $E(K_{n,n} \setminus G) \subseteq \{x_i y_i : \, 1 \leq i \leq n\}$.
Moreover, assume without loss of generality that the edge claimed by Breaker in his first
move is either $x_1 y_1$ or $x_1 y_2$. Let $A_1 = \{x_1, \ldots, x_{\left\lceil n/2 \right\rceil}\}$, 
$A_2 = \{y_{\left\lfloor n/2 \right\rfloor + 1}, \ldots, y_n\}$, $B_1 = U_1 \setminus A_1$,
$B_2 = U_2 \setminus A_2$, $H'_1 = (G \setminus B)[A_1 \cup A_2]$ and $H_2 = (G \setminus B)[B_1 \cup B_2]$. 
Note that $H_2 \cong K_{|B_1|, |B_1|}$ and that there exists an edge $e \in E(K_{|A_1|, |A_1|})$
such that $H'_1 \supseteq K_{|A_1|, |A_1|} \setminus \{e\}$. Let $H_1 = K_{|A_1|, |A_1|} 
\setminus \{e\}$ (if $H'_1 = K_{|A_1|, |A_1|}$, then choose $e \in E(K_{|A_1|, |A_1|})$ arbitrarily). 
Let ${\mathcal S}_1$ (respectively ${\mathcal S}_2$) be Maker's strategy for 
the perfect matching game on $K_{|A_1|, |A_1|}$ (respectively $K_{|B_1|, |B_1|}$) whose existence 
follows from Theorem~\ref{th::PM}. Maker plays her first move in $H_1$ according to ${\mathcal S}_1$. 
She views the board to be $E(K_{|A_1|, |A_1|})$ and assumes that Breaker claimed $e$ in his first move.   
In the remainder of the game, Maker plays on $E(H_1)$ and $E(H_2)$ in parallel. 
That is, whenever Breaker claims an edge of $H_i$ for some $i \in \{1,2\}$, Maker claims a 
free edge of the same board according to ${\mathcal S}_i$ (unless she has already built a 
perfect matching on this board, in which case she claims a free edge of the other board) and
whenever Breaker claims an edge of $G \setminus (H_1 \cup H_2)$, Maker plays in some $H_i$ in 
which she has not yet built a perfect matching. 

Since Maker plays according to ${\mathcal S}_1$ and ${\mathcal S}_2$, it follows by Theorem~\ref{th::PM} 
that she builds a perfect matching of $H_1$ within $|A_1| + 1$ moves and a perfect matching 
of $H_2$ within $|B_1| + 1$ moves. The union of these two matchings forms a perfect matching 
of $G$ which Maker builds within $n+2$ moves.

Assume then that $k_1 + k_2 \geq 3$ and that the assertion of the lemma holds 
for $k_1 + k_2 - 1$. Assume without loss of generality that $k_2 \geq k_1$; 
in particular, $k_2 \geq 2$. We present a strategy for Maker and then prove that it
allows her to build a perfect matching of $G$ within $n+2$ moves. At any point during 
the game, if Maker is unable to follow the proposed strategy, then she forfeits the game. 
The strategy is divided into the following two stages.

\textbf{Stage I:} Maker builds a matching while making sure that neither $\Delta_1$
nor $\Delta_2$ are increased and trying to decrease $\Delta_1 + \Delta_2$. This stage is
divided into the following two phases.

\textbf{Phase 1:} At the beginning of the game and immediately after each of her moves in this phase,
if $\Delta_1 < k_1$, then Maker proceeds to Stage II. Otherwise, if there exists a free edge $uv$ 
such that  
\begin{description}
\item [(a)] $u \in S_1$ and $v \in S_2$;
\item [(b)] $d_{Br}(u) = \Delta_1$;
\item [(c)] $d_{Br}(v) = \max \{d_{Br}(w) : \, w \in N_G(u, S_2) \textrm{ for which } uw \textrm{ is free}\}$;
\item [(d)] $d_{Br}(v) \geq 2$;
\end{description}
then Maker claims an arbitrary such edge and repeats Phase 1. If no such edge exists, then Maker proceeds to
Phase 2.

\textbf{Phase 2:} In her first move in this phase, Maker claims a free edge $uv$ such that
$u \in S_1$, $d_{Br}(u) = \Delta_1$ and $v \in S_2$. Let $xy$ denote the edge claimed by 
Breaker in his following move, where $x \in U_1$ and $y \in U_2$. In her next (and final) 
move in this phase, Maker plays as follows.  
\begin{description}
\item [(a)] If $x \notin S_1$ or $y \notin S_2$, then Maker claims a free edge $ab$ such that $a \in S_1$, 
$b \in S_2$ and $d_{Br}(b) = \Delta_2$. 
\item [(b)] Otherwise, if $d_{Br}(y) > k_2$, then Maker claims a free edge $yz$ for an arbitrary vertex $z \in N_G(y, S_1)$.
\item [(c)] Otherwise, if there exists a vertex $w \in S_2$ such that $d_{Br}(w) \geq k_2$ and
$xw$ is free, then Maker claims $xw$. 
\item [(d)] Otherwise, Maker claims a free edge $xz$ for an arbitrary vertex $z \in N_G(x, S_2)$.
\end{description}
Maker then proceeds to Stage II. 

\textbf{Stage II:} Maker builds a perfect matching of $G[S]$ within $|S_1| + 2$ moves.

It is evident that, if Maker can follow the proposed strategy without forfeiting the game, 
then she wins the perfect matching game, played on $E(G)$, within $n+2$ moves. 
It thus suffices to prove that she can indeed do so. 

We begin by proving the following simple claim.

\begin{claim} \label{cl::keepDegreesLow}
If Maker follows the proposed strategy, then $\Delta_1 \leq k_1$ and $\Delta_2 \leq k_2$ 
hold immediately after each of Maker's moves in Phase 1 of Stage I.
\end{claim}

\begin{proof}
The claim clearly holds before the game starts. Assume it holds immediately after Maker's $j$th move 
for some non-negative integer $j$. Let $xy$ denote the edge claimed by Breaker in his $(j+1)$st move, where 
$x \in U_1$ and $y \in U_2$. Since Maker does not increase $d_{Br}(w)$ for any $w \in S$ in any of her moves,
it follows that if $x \notin S_1$ or $y \notin S_2$, then there is nothing to prove. Assume then
that $x \in S_1$ and $y \in S_2$. It follows by our assumption that $\Delta_1 \leq k_1 + 1$ and $\Delta_2 \leq k_2 + 1$
and that $d_{Br}(w) \leq k_1$ holds for every $w \in S_1 \setminus \{x\}$ and $d_{Br}(w) \leq k_2$ 
holds for every $w \in S_2 \setminus \{y\}$. Let $uv$ denote the edge claimed by Maker in her $(j+1)$st 
move, where $u \in S_1$ and $v \in S_2$. If $u = x$, then $x$ is removed from $S_1$ and,
as a result, $d_{Br}(y) \leq k_2$ holds after this move. Assume then that $u \neq x$; it follows
by Maker's strategy that $d_{Br}(x) \leq d_{Br}(u) \leq k_1$. If $d_{Br}(y) \leq k_2$, then there
is nothing to prove. Assume then that $d_{Br}(y) = k_2 + 1$. If $v = y$, then $y$ is removed from 
$S_2$. Assume then that $v \neq y$. Since $y$ is the unique vertex of maximum degree in $S_2$,
it follows by Maker's strategy that $uy \in E(Br)$. Hence, by claiming $uv$ Maker decreases $d_{Br}(y)$.
We conclude that $\Delta_1 \leq k_1$ and $\Delta_2 \leq k_2$ hold immediately after Maker's
$(j+1)$st move. 
\end{proof}

We will first prove that Maker can follow Stage I of her strategy without forfeiting the game, 
and, moreover, that this stage lasts at most $\frac{k_1 n}{k_1 + 1} + 2$ moves. 

It is obvious that Maker can follow her strategy for Phase 1. We will prove that this phase
lasts at most $\frac{k_1 n}{k_1 + 1}$ moves. For every non-negative integer $i$, immediately 
after Breaker's $(i+1)$st move, let $D(i) = \sum_{v \in S_1} d_{Br}(v)$. Note that $D(i) \geq 0$ 
holds for every $i$ and that $D(0) \leq k_1 n + 1$. For an arbitrary non-negative integer $j$, 
let $uv$ be the edge claimed by Maker in her $(j+1)$st move. Then $D(j+1) \leq D(j) - d_{Br}(u) - 
d_{Br}(v) + 1 \leq D(j) - (k_1 + 1)$, where the last inequality follows by Properties (b) and (d) 
of the proposed strategy for Phase 1. It follows that there can be at most $\frac{k_1 n}{k_1 + 1}$ 
such moves throughout Stage I. 

By its description, Phase 2 lasts exactly 2 moves. It follows that indeed Stage I lasts at most
$\frac{k_1 n}{k_1 + 1} + 2$ moves. Therefore, $|S_1| = |S_2| \geq \frac{n}{k_1 + 1} - 2 > 2 +
\max \{k_1, k_2\} \geq \max \{\Delta_1, \Delta_2\}$ holds throughout Stage I, where the second 
inequality holds since $n$ is sufficiently large with respect to $k_1$ and $k_2$. Hence, for every 
$u \in S$ there exists some $v \in S$ such that $uv \in E$ is free. In particular, Maker can follow
the proposed strategy for Phase 2.  

It remains to prove that Maker can follow Stage II of the proposed strategy without forfeiting
the game. Consider the game immediately after Maker's last move in Stage I (or before the game 
starts in case Maker plays no moves in Stage I). As noted above, at this point we have
$|S_1| = |S_2| \geq \frac{n}{k_1 + 1} - 2 \geq \max \{f(k_1 - 1, k_2), f(k_1, k_2 - 1)\}$, where 
the last inequality holds for sufficiently large $n$. 

We claim that $\Delta_1 \leq k_1$, $\Delta_2 \leq k_2$ and $\Delta_1 + \Delta_2 \leq k_1 + k_2 - 1$ 
hold at this point as well. Note that, by Claim~\ref{cl::keepDegreesLow}, $\Delta_1 \leq k_1$ and 
$\Delta_2 \leq k_2$ hold after each of Maker's moves in Phase 1 of Stage I. If Maker enters Stage II directly 
from Phase 1 of Stage I, then $\Delta_1 < k_1$ holds as well and our claim follows. Assume then that Maker plays 
the two moves of Phase 2. It follows by Claim~\ref{cl::keepDegreesLow} that immediately before Maker's
first move in this phase there is at most one vertex $z \in S_1$ such that $d_{Br}(z) > k_1$ and
at most one vertex $z' \in S_2$ such that $d_{Br}(z') > k_2$. In her first move in Phase 2, Maker claims
an edge $uv$ such that $d_{Br}(u) = \Delta_1$. Since this is done in Phase 2, it follows that $uw \in E(Br)$ 
holds at this moment for every $w \in S_2$ for which $d_{Br}(w) \geq 2$. Clearly, $\Delta_1 \leq k_1$ holds 
after this move. Moreover, since $k_2 \geq 2$, by removing $u$ from $S_1$, Maker decreases $d_{Br}(w)$
for every $w \in S_2$ whose degree was at least $k_2$. Hence, $\Delta_2 \leq k_2$ holds after this move
and, moreover, there is at most one vertex $z'' \in S_2$ such that $d_{Br}(z'') = k_2$. In his next move,
Breaker claims an edge $xy$. It is not hard to see that each of the four options for Maker's next move 
(as described in the proposed strategy), ensures that $\Delta_1 \leq k_1$ and $\Delta_2 < k_2$ will hold
after this move.  

We conclude that $|S_1| = |S_2| \geq \max \{f(k_1 - 1, k_2), f(k_1, k_2 - 1)\}$, $\Delta_1 \leq k_1$, 
$\Delta_2 \leq k_2$ and $\Delta_1 + \Delta_2 \leq k_1 + k_2 - 1$ hold immediately before Breaker's 
first move in Stage II. It thus follows by the induction hypothesis that Maker can indeed build a 
perfect matching of $G[S]$ within $|S_1| + 2$ moves. 
{\hfill $\Box$ \medskip\\}

We are now ready to prove the main result of this section.

\textbf{Proof of Theorem~\ref{th::manyLeaves}}
Let $L$ denote the set of leaves of $T$ and let $\varepsilon = (2 \Delta (m_2 + 1))^{-1}$. Since 
$\Delta(T) \leq \Delta$ and since $T$ does not admit a bare path of length $m_2$, it follows by 
Lemma~\ref{lem::NoPathLargeMatching} that $|L| \geq |N_T(L)| \geq \frac{n}{2 \Delta (m_2 + 1)}
= \varepsilon n$. Let $L' \subseteq L$ be a maximal set of leaves, no two of which have a common 
parent in $T$ (that is, $|L'| = |N_T(L)|$) and let $T' = T \setminus L'$.  

First we describe a strategy for Maker in $(E(K_n), {\mathcal T}_n)$ and then prove that it allows her to 
build a copy of $T$ within $n+1$ moves. At any point during the game, if Maker is unable to follow the proposed 
strategy, then she forfeits the game. The proposed strategy is divided into the following two stages.

\textbf{Stage I:} In this stage, Maker's aim is to embed a tree $T''$ such that 
$T'\subseteq T'' \subseteq T$ and $|V(T'')| \leq n - \varepsilon n/2$. 
Moreover, Maker does so in exactly $|V(T'')| - 1$ moves.

Let $k$ be the smallest integer such that 
$\Delta + 3 \leq \varepsilon \Delta^k /40$. Throughout this stage, Maker maintains a set 
$S \subseteq V(T)$ of embedded vertices, an $S$-partial embedding $f$ of $T$ in $K_n \setminus B$, 
a set $A = V(K_n) \setminus f(S)$ of available vertices and a set $D \subseteq V(K_n)$ of \emph{dangerous} 
vertices, where a vertex $v \in V(K_n)$ is called dangerous if $d_B(v) \geq \Delta^{k+1}$ and $v$ is 
either an available vertex or an open vertex with respect to $T$. Initially, $D = \emptyset$, 
$S = \{v'\}$ and $f(v') = v$, where $v' \in V(T')$ and $v \in V(K_n)$ are arbitrary vertices. 

For as long as $V(T') \setminus S \neq \emptyset$ or $D \neq \emptyset$,
Maker plays as follows:
\begin{enumerate} [(1)]
\item \label{I1} If $D \neq \emptyset$, then let $v \in D$ be an arbitrary vertex. We distinguish between 
the following two cases:
\begin{enumerate}[(i)]

\item \label{I1i} \textbf{$v$ is taken.} Let $v'_1, \ldots, v'_r$ be the new neighbors
of $v' := f^{-1}(v)$ in $T$. In her next $r$ moves, Maker claims the edges of $\{v v_i : 1 \leq i \leq r\}$, 
where $v_1, \ldots, v_r$ are $r$ arbitrary available vertices. Subsequently, Maker updates $S$, $D$ and $f$ by 
adding $v'_1, \ldots, v'_r$ to $S$, deleting $v$ from $D$ and setting $f(v'_i) = v_i$ for every $1 \leq i \leq r$.

\item \label{I1ii} \textbf{$v$ is available.} This case is further divided into the following three subcases:
\begin{enumerate}[(a)]
\item \label{I1iia} There exists a vertex $u \in f({\mathcal O}_T)$ such that the edge $uv$ is free. 
Maker claims $uv$ and updates $S$ and $f$ by adding $v'$ to $S$ and setting $f(v') = v$, where 
$v' \in N_T(f^{-1}(u))$ is an arbitrary new vertex. If $v'$ is a leaf of $T$, then Maker deletes $v$
from $D$.

\item \label{I1iib} There are two vertices $u, w \in f({\mathcal O}_T)$ and new vertices $u_1, u_2, w_1, w_2
\in V(T) \setminus S$ such that $f^{-1}(u) u_1, u_1 u_2, f^{-1}(w) w_1, w_1 w_2 \in E(T)$.
Let $z$ be an available vertex such that the edges $zv$, $zu$ and $zw$ are free. Maker claims the edge $zv$ and
after Breaker's next move she claims $zu$ if it is free and $zw$ otherwise. Assume that Maker claims $zu$ (the
complementary case in which she claims $zw$ is similar). She then updates $S$ and $f$ by adding $u_1$ and $u_2$ 
to $S$ and setting $f(u_1) = z$ and $f(u_2) = v$. If $u_2$ is a leaf of $T$, then Maker deletes $v$ from $D$.

\item \label{I1iic} There exists a vertex $u \in f({\mathcal O}_T)$ and new vertices 
$x', y', z' \in V(T) \setminus S$ such that $f^{-1}(u) x', x'y', y'z' \in E(T)$.
Maker claims a free edge $vw$ for some $w \in A$. Immediately after Breaker's next move,
let $x$ be an available vertex such that the edges $xu$, $xv$ and $xw$ are free. Maker 
claims the edge $xu$ and after Breaker's next move she claims $xw$ if it is free and $xv$ 
otherwise. Assume that Maker claims $xw$ (the complementary case in which she claims $xv$ 
is similar). She then updates $S$ and $f$ by adding $x'$, $y'$ and $z'$ to $S$ and setting
$f(x') = x$, $f(y') = w$ and $f(z') = v$. If $z'$ is a leaf of $T$, then Maker deletes $v$
from $D$.
\end{enumerate}
\end{enumerate}
\item \label{I2} If $D = \emptyset$, then Maker claims an arbitrary edge $uv$, where 
$u \in f({\mathcal O}_{T'})$ and $v \in A$. Subsequently, she updates $S$ and $f$ by 
adding $v'$ to $S$ and setting $f(v') = v$, where $v' \in N_{T'}(f^{-1}(u))$ is an
arbitrary new vertex.  
\end{enumerate}

As soon as $V(T') \setminus S = D = \emptyset$, Stage I is over and Maker proceeds to Stage II.

\textbf{Stage II:} Let $H$ be the bipartite graph with parts $A$ and $f({\mathcal O}_T)$ and
edge set $E(H) = \{uv \in E(K_n) \setminus E(B): u \in A, \, v \in f({\mathcal O}_T)\}$. 
Maker builds a perfect matching of $H$ within $|A| + 2$ moves, following the strategy whose 
existence is ensured by Lemma~\ref{FastPerMatBP}.

It is evident that if Maker can follow the proposed strategy without forfeiting the game, then 
she wins the game within $n+1$ moves. It thus suffices to prove that Maker can indeed do so. 
We consider each of the two stages separately.

\textbf{Stage I:} We begin by proving the following three claims.

\begin{claim} \label{claim1}
At most $\frac{2n}{\Delta^{k+1}}$ vertices become dangerous throughout Stage I.
\end{claim}

\begin{proof} Stage I of the proposed strategy lasts $|V(T'')| - 1 \leq n$ moves.
Since, moreover, a dangerous vertex has degree at least $\Delta^{k+1}$ in Breaker's graph, 
it follows that there can be at most $\frac{2n}{\Delta^{k+1}}$ such vertices.
\end{proof}

\begin{claim} \label{claim2}
The following two properties hold at any point during Stage I.
\begin{enumerate}[$(1)$]
\item $|A| \geq \varepsilon n/2$;
\item $d_B(v) \leq \varepsilon n/(10 \Delta)$ holds for every vertex $v \in A \cup f({\mathcal O}_T)$.
\end{enumerate}
\end{claim}

\begin{proof} Starting with $(1)$, note that $|A| = n - |S|$ and that $|S| = |V(T')| + |L' \cap S|$ 
holds at the end of Stage I. Since $|V(T')| \leq n - \varepsilon n$ it suffices to prove that 
$|L' \cap S| \leq \varepsilon n/2$. Let $w' \in L' \cap S$ be an arbitrary vertex and let $w = f(w')$. 
Since Maker follows the proposed strategy, $D \cap \{w, f(N_T(w'))\} \neq \emptyset$ must have been 
true at some point during Stage I. Using Claim~\ref{claim1} we conclude that 
$$
|L' \cap S| \leq \frac{2n}{\Delta^{k+1}} \leq \frac{\varepsilon n}{2} \,.
$$

Next, we prove $(2)$. Let $v \in A \cup f({\mathcal O}_T)$ be an arbitrary vertex.
If $v$ was never a dangerous vertex, then $d_B(v) < \Delta^{k+1} \leq \varepsilon n/(10 \Delta)$
holds by definition and since $n$ is sufficiently large with respect to $\Delta$ and $k$. Otherwise, for 
as long as $v \in D$, Maker plays according to Case (1) of the proposed strategy. Therefore, unless 
Maker forfeits the game, at some point during Stage I she connects $v$ to her tree (this requires 
zero moves in Case (i), one move in Case (ii)(a), two moves in Case (ii)(b) and three moves in Case (ii)(c)). 
Since $v$ can be removed from $D$ only in Case (i) or if $f^{-1}(v)$ is a leaf of $T$, it follows that, 
unless Maker forfeits the game, at some point during Stage I she closes $v$. According to the proposed 
strategy for Case (i), this requires at most $\Delta$ moves. We conclude that Maker spends at most 
$\Delta + 3$ moves on connecting a dangerous vertex to her tree and closing it. It thus follows by 
Claim~\ref{claim1} that 
$$
d_B(v) \leq \Delta^{k+1} + (\Delta + 3) \cdot \frac{2n}{\Delta^{k+1}} \leq 
\Delta^{k+1} + \frac{\varepsilon \Delta^k}{40} \cdot \frac{2n}{\Delta^{k+1}} 
\leq \frac{\varepsilon n}{10 \Delta} \,,
$$
where the last inequality holds since $n$ is sufficiently large with respect to $\Delta$ and $k$.
\end{proof}

\begin{claim} \label{claim3}
At any point during Stage I, if $D \neq \emptyset$ and $v \in D$ is available, then at least one of the conditions (a), (b) or (c) of
Case (1)(ii) must hold.
\end{claim}
\begin{proof}
Suppose for a contradiction that none of (a), (b) and (c) hold. Since (a) does not hold and since 
$d_B(v) \leq \varepsilon n/(10 \Delta)$ holds by Part (2) of Claim~\ref{claim2}, it follows that 
$|N_T(L) \cap {\mathcal O}_T| \leq |{\mathcal O}_T| \leq \varepsilon n/(10 \Delta)$. Since (b) does not hold, 
it follows that $|{\mathcal O}_T \setminus N_T(L)| \leq 1$. Finally, since (c) does not hold, it follows that 
if $x \in {\mathcal O}_T \setminus N_T(L)$, then $x \in N_T(N_T(L))$. Therefore 
\begin{eqnarray*}
|A| &\leq& |N_T(L) \cap {\mathcal O}_T| \cdot \Delta + |{\mathcal O}_T \setminus N_T(L)| \cdot (\Delta + \Delta^2) \\
&\leq& \varepsilon n/(10 \Delta) \cdot \Delta + 1 \cdot (\Delta + \Delta^2) \\ 
&<& \varepsilon n/2 \,,
\end{eqnarray*}
contrary to Part $(1)$ of Claim~\ref{claim2}.
\end{proof}

Next, we consider each case of Stage I separately and prove that Maker
can follow the proposed strategy for that case.
\begin{enumerate}[$(1)$]
\item In this case $D \neq \emptyset$. Let $v \in D$ be an arbitrary vertex.
\begin{enumerate}[$(i)$]
\item For as long as $v$ is open we have $d_B(v) \leq \varepsilon n/(10 \Delta) < \varepsilon n/2 - 2\D \leq |A| - 2\D$,
where the first inequality holds by Part (2) of Claim~\ref{claim2} and the last inequality 
holds by Part (1) of Claim~\ref{claim2}. Maker can thus close $v$ as instructed by the proposed strategy
for this case.  

\item In this case (and all of its subcases) $v$ is available.
\begin{enumerate} [$(a)$]
\item It readily follows by its description that Maker can follow the proposed strategy
for this subcase.
\item Let $u$ and $w$ be open vertices as described in the proposed strategy for this subcase.
It follows by Parts (1) and (2) of Claim~\ref{claim2} that 
$$
d_B(v) + d_B(u) + d_B(w) \leq 3 \varepsilon n/(10 \Delta) < \varepsilon n/2 \leq |A| \,.
$$
We conclude that there exists a vertex $z \in A$ such that the edges $zv$, $zu$ and $zw$
are free.

\item Similarly to Case (i) above, there exists a vertex $w \in A$ such that the edge $vw$ is free.
Similarly to case (ii)(b) above, there exists a vertex $z \in A$ such that the edges $zv$, $zu$ 
and $zw$ are free.
\end{enumerate}

\end{enumerate}

\item Since $D = \emptyset$ and yet Stage I is not over, it follows that $V(T') \setminus S
\neq \emptyset$. It follows that $\open_{T'} \neq \emptyset$. Let $u \in f(\open_{T'})$ be an 
arbitrary vertex. Since $D = \emptyset$, it follows that $d_B(u) < \Delta^{k+1} < \varepsilon n/2
\leq |A|$, where the last inequality follows from Part (1) of Claim~\ref{claim2}.
We conclude that there exists a vertex $v \in A$ such that $uv$ is free. 
\end{enumerate}

\textbf{Stage II:} Since $D = \emptyset$ holds at the end of Stage I, it follows
that $\delta(H) \geq |A| - \Delta^{k+1}$. Since, moreover, $n$ is sufficiently large
and $|A| \geq \varepsilon n/2$ holds by Part (1) of Claim~\ref{claim2}, it follows
by Lemma~\ref{FastPerMatBP} that Maker has a strategy to win the perfect matching game,
played on $E(H)$, within $|A| + 2$ moves.  

At the end of Stage I, Maker's graph is a tree isomorphic to $T''$. Hence,
Stage I lasts exactly $|V(T'')| - 1$ moves. By Lemma~\ref{FastPerMatBP}, Stage II
lasts at most $|A| + 2 = |V(T)| - |V(T'')| + 2$ moves. We conclude that the entire  
game lasts at most $|V(T)| + 1 = n + 1$ moves.
{\hfill $\Box$ \medskip\\}

\section{Building trees in optimal time} \label{sec::RandomTree}

In this section we will prove Theorems~\ref{th::randomTree} and~\ref{th::PathFromLeaf}. A central ingredient in the proofs 
of both theorems is Maker's ability to build a Hamilton path with some designated vertex as an endpoint in optimal time. 
Our strategy for building a path quickly is based on the proof of Theorem 1.4 from~\cite{HKSS}. In particular, the
first step is to build a perfect matching. 

\begin{lemma} \label{lem::MissingEdgesPM}
For every sufficiently large integer $r$ there exists an integer $n_0 = n_0(r)$ such that 
for every even integer $n \geq n_0$ and every graph $G$ with $n$ vertices and
$e(G) \geq \binom{n}{2} - n + r$ edges, Maker has a strategy to win the perfect matching 
game, played on $E(G)$, within $n/2 + 1$ moves.
\end{lemma} 

\begin{proof}
The following notation and terminology will be used throughout this proof. 
At any point during the game, let $S$ denote the set of vertices of $G$ 
which are isolated in Maker's graph. Let $Br = ((K_n \setminus G) \cup B)[S]$. 
For every free edge $e \in G[S]$, let $D(e) = |\left\{f \in E(Br) : e \cap f \neq \emptyset \right\}|$
denote the \emph{danger} of $e$.

We present a strategy for Maker and then prove that it
allows her to build a perfect matching of $G$ within $n/2 + 1$ moves. At any point during 
the game, if Maker is unable to follow the proposed strategy, then she forfeits the game. 
The strategy is divided into the following two stages.

\textbf{Stage I:} If there exists a free edge $e \in G[S]$ such that $D(e) \geq 3$,
then Maker claims an arbitrary such edge and repeats Stage I. Otherwise, she proceeds
to Stage II. 

\textbf{Stage II:} Maker builds a perfect matching of $G[S]$ within $|S|/2 + 1$ moves. 

It is evident that, if Maker can follow the proposed strategy without forfeiting the game, 
then she wins the perfect matching game, played on $E(G)$, within $n/2 + 1$ moves. 
It thus suffices to prove that she can indeed do so. 

It is clear by its description that Maker can follow Stage I of the proposed strategy
without forfeiting the game. In order to prove that she can also follow Stage II of the 
proposed strategy, we first prove the following three claims.

\begin{claim} \label{cl::freeEdgesStageI}
$e(Br) \leq v(Br) - 2$ holds at any point during Stage I.
\end{claim}

\begin{proof}
The required inequality holds before and immediately after Breaker's first move since $e(Br) \leq e(K_n \setminus G) + 1 
\leq n - r + 1 \leq n - 2 = v(Br) - 2$ holds at that time, where the second inequality holds by assumption and the third 
inequality holds since $r \geq 3$. Assume that this inequality holds immediately after Breaker's $j$th move for
some positive integer $j$. If Maker plays her $j$th move in Stage I, then she claims an edge $e \in G[S]$
such that $D(e) \geq 3$. This decreases $v(Br) = |S|$ by $2$ and $e(Br)$ by at least $3$. It follows that
$e(Br) \leq v(Br) - 3$ holds immediately after Maker's $j$th move. In his $(j+1)$st move, Breaker increases
$e(Br)$ by at most $1$ and does not decrease $v(Br)$. Hence $e(Br) \leq v(Br) - 2$ holds immediately after 
his $(j+1)$st move.  
\end{proof}

\begin{claim} \label{cl::ShortStageI}
Maker plays at most $(n-r)/2$ moves in Stage I.
\end{claim}

\begin{proof}
In each round (that is, a move of Maker and a counter move of Breaker) of Stage I, 
$e(Br)$ is decreased by at least $2$ (it is decreased by $D(e) \geq 3$ in Maker's 
move and then increased by at most $1$ in Breaker's move). The claim now follows since
$e(Br) \geq 0$ holds at any point during the game and $e(Br) \leq e(K_n \setminus G) + 1 
\leq n - r + 1$ holds immediately after Breaker's first move. 
\end{proof}

\begin{claim} \label{cl::partition}
Let $m \geq 6$ be an even integer and let $H = (V,E)$ be a graph on $m$ vertices which satisfies the 
following two properties:
\begin{description}
\item [(i)] $|\{f \in E : e \cap f \neq \emptyset\}| \leq 2$ for every $e \in E(K_m) \setminus E$.
\item [(ii)] For every $u \in V$ there exists a vertex $v \in V$ such that $uv \notin E$.
\end{description}
Then there exists a partition $V = A \cup B$ such that $|A| = |B| = m/2$ and $e_H(A,B) \leq 1$. 
\end{claim}

\begin{proof}
Note that $\Delta(H) \leq 2$. Indeed, suppose for a contradiction that there exist
vertices $u, v_1, v_2, v_3 \in V$ such that $u v_1, u v_2, u v_3 \in E$. It follows by Property
(ii) that there exists a vertex $v_4 \in V$ such that $u v_4 \notin E$. We thus have 
$u v_1, u v_2, u v_3 \in \{f \in E : u v_4 \cap f \neq \emptyset\}$, contrary to Property (i). 

Assume first that $\Delta(H) = 2$ and let $u, v, w \in V$ be such that $uv, uw \in E$. Let $A$
be an arbitrary subset of $V \setminus \{u,v,w\}$ of size $m/2$ (such a set $A$ exists since 
$m \geq 6$) and let $B = V \setminus A$. We claim that $e_H(A,B) = 0$. Indeed, suppose for a 
contradiction that there exist vertices $x \in A$ and $y \in B$ such that $xy \in E$. 
Since $\Delta(H) \leq 2$ and $uv, uw \in E$, it follows that $ux \in E(K_m) \setminus E$. 
However, we then have $uv, uw, xy \in \{f \in E : u x \cap f \neq \emptyset\}$, contrary to 
Property (i).

Assume then that $\Delta(H) \leq 1$, that is, $H$ is a matching. Let $E = \{x_i y_i : 1 \leq i \leq \ell\}$,
where $0 \leq \ell \leq m/2$ is an integer. Let $A = \{x_1, \ldots, x_{\lceil m/4 \rceil}, y_1, \ldots, y_{\lfloor m/4 \rfloor}\}$ 
and let $B = V \setminus A$. Note that $|A| = |B| = m/2$ and that $E_H(A,B) \subseteq 
\{x_{\lceil m/4 \rceil} y_{\lceil m/4 \rceil}\}$ and thus $e_H(A,B) \leq 1$ as claimed.  
\end{proof}

We are now ready to prove that Maker can follow Stage II of the proposed strategy without forfeiting the game.
It follows by the description of Stage I of the proposed strategy that $D(e) \leq 2$ holds for every free edge 
$e \in G[S]$ at the beginning of Stage II. Moreover, it follows by Claim~\ref{cl::freeEdgesStageI} that, immediately 
after Breaker's last move in Stage I, for every $u \in V$ there is a free edge $e$ such that $u \in e$. 
Therefore, the conditions of Claim~\ref{cl::partition} are satisfied (with $H = Br$). Hence, there exists 
a partition $S = A \cup B$ such that $e_{Br}(A,B) \leq 1$. Let $e$ be an edge for which $E_{G[S]}(A,B) \supseteq
E_{K_n}(A,B) \setminus \{e\}$. Maker (being the first to play in Stage II) plays the perfect matching game on
$E_{K_n}(A,B) \setminus \{e\}$. She pretends that she is in fact playing as the second player on $E_{K_n}(A,B)$
and that Breaker has claimed $e$ in his first move. Since $r$ is sufficiently large and $|S| \geq n - 2 (n-r)/2 = r$ 
holds by Claim~\ref{cl::ShortStageI}, it follows by Theorem~\ref{th::PM} that Maker has a strategy to win the perfect 
matching game, played on $E_{K_n}(A,B)$, within $|S|/2 + 1$ moves.
\end{proof}

We will use Lemma~\ref{lem::MissingEdgesPM} to prove the following result.

\begin{lemma} \label{lem::PathOneEndpoint}
There exists an integer $m_0$ such that the following holds for every $m \geq m_0$. 
Let $G$ be a graph with $m$ vertices and $\binom{m}{2} - k$ edges, where $k$ is a non-negative integer. Assume
that $k \leq (m - 25)/2$ if $m$ is odd and $k \leq (m - 28)/2$ if $m$ is even. Let $x$ be an arbitrary vertex 
of $G$. Then, playing a Maker-Breaker game on $E(G)$, Maker has a strategy to build in $m - 1$ moves a Hamilton 
path of $G$ such that $x$ is one of its endpoints. 
\end{lemma}

\begin{proof}
The following notation and terminology will be used throughout this proof. Given paths $P_1 = (v_1 \ldots v_t)$ and 
$P_2 = (u_1 \ldots u_r)$ in a graph $G$ for which $v_t u_1 \in E(G)$, let $P_1 \circ v_t u_1 \circ P_2$ denote the 
path $(v_1 \ldots v_t u_1 \ldots u_r)$. Let $G$ be a graph on $m$ vertices and let $P_0, P_1, \ldots, P_{\ell}$ 
be paths in $G$ where $P_0 = \{p_0\}$ is a special path of length zero and $e(P_i) \geq 1$ for every $1 \leq i \leq \ell$. 
For every $1 \leq i \leq \ell$ let $End(P_i)$ denote the set of two endpoints of the path $P_i$ and let $End = 
\bigcup_{i=1}^{\ell} End(P_i) \cup \{p_0\}$. Let 
$$
X = \left\{uv \in E(K_m) : \{u,v\} \in \binom{End}{2} \textrm{ and } \{u,v\} \neq End(P_i) \textrm{ for every } 
1 \leq i \leq \ell \right\} \,.
$$ 
At any point during the game, let $Br$ denote the graph with vertex set $End$ and edge set
$X \cap (E(K_m \setminus G) \cup E(B))$. The edges of $X \setminus E(Br)$ are called \emph{available}. 
For every available edge $e$, let  $D(e) = |\left\{f \in E(Br) : e \cap f \neq \emptyset \right\}|$ 
denote the \emph{danger} of $e$.

Without loss of generality we can assume that $m$ is odd (otherwise, in her first move, Maker claims an arbitrary free 
edge $x x'$ and then plays on $(G \setminus B)[V(G) \setminus \{x\}]$ with $x'$ as the designated endpoint; note that 
$k \leq (m - 28)/2 \Longrightarrow k+1 \leq [(m-1) - 25]/2$). 

We present a strategy for Maker and then prove that it allows her to build the required path in $m - 1$ moves. 
At any point during the game, if Maker is unable to follow the proposed strategy, then she forfeits the game. 
The strategy is divided into the following five stages.

\textbf{Stage I:} Maker builds paths $P_1, \ldots, P_{(m-3)/2}$ in $G \setminus \{x\}$ which satisfy the following three properties:
\begin{description}
\item [(a)] $e(P_1) = 3$.
\item [(b)] $e(P_i) = 1$ for every $2 \leq i \leq (m-3)/2$.
\item [(c)] $V(P_i) \cap V(P_j) = \emptyset$ for every $1 \leq i < j \leq (m-3)/2$. 
\end{description}
This stage lasts exactly $(m-1)/2 + 1$ moves. As soon as it is over, Maker proceeds to Stage II.

\textbf{Stage II:} Let $p_0 = x$, let $P_0 = \{p_0\}$, let $\ell = (m-3)/2$ and let ${\mathcal P} = \{P_0, P_1, \ldots, P_{\ell}\}$. 
For every $i \geq (m-1)/2 + 2$, immediately before her $i$th move, Maker checks
whether there exists an available edge $e \in X \setminus E(Br)$ such that $D(e) \geq 3$. If there is 
no such edge, then this stage is over and Maker proceeds to Stage III. Otherwise, in her $i$th move, Maker claims an 
arbitrary such edge $uv$. She then updates ${\mathcal P}$ as follows. Let $0 \leq i < j \leq \ell$ denote the unique
indices for which $u \in V(P_i)$ and $v \in V(P_j)$. Maker deletes $P_j$ from ${\mathcal P}$. Moreover, If $i \geq 1$, then 
she replaces $P_i$ with $P_i \circ uv \circ P_j$ (which is now referred to as $P_i$) and if $i = 0$, then she sets 
$p_0 = z$, where $z$ is the unique vertex in $End(P_j) \setminus \{v\}$. In both cases the set $X$ is updated accordingly.    

\textbf{Stage III:} If $\Delta(Br) \leq 1$, then this stage is over and Maker proceeds to Stage IV.  
Otherwise, she claims an available edge $u u'$, where $u \in End$ is an arbitrary vertex of degree at least 2 in $Br$.
Maker then updates ${\mathcal P}$ and $X$ as in Stage II and repeats Stage III. 

\textbf{Stage IV:} In her first move in this stage, Maker plays as follows. If there exists a vertex $w \in End$
such that $p_0 w \in E(Br)$, then Maker claims an available edge $wz$. Otherwise, she claims an arbitrary available edge.
In either case she updates ${\mathcal P}$ and $X$ as in Stage II. 

For every $i \geq 2$, before her $i$th move in this stage, Maker checks how many paths are in ${\mathcal P}$. If there are 
exactly 3 paths, then this stage is over and she proceeds to Stage V; otherwise, she plays as follows. Let $uv$ denote
the edge claimed by Breaker in his last move; assume without loss of generality that $u \neq p_0$. If $uv \notin X$,
then Maker claims an arbitrary available edge. Otherwise she claims an available edge $uw$ for some $w \in End
\setminus \{p_0\}$. In either case Maker updates ${\mathcal P}$ and $X$ as in Stage II and repeats Stage IV.

\textbf{Stage V:} Claiming two more edges, Maker connects her 3 paths to a Hamilton path of $G$ such that $x$ is
one of its endpoints.

It is evident that, if Maker can follow the proposed strategy without forfeiting the game, 
then she builds a Hamilton path of $G$ such that $x$ is one of its endpoints in $m - 1$ moves. 
It thus suffices to prove that she can indeed do so. We consider each stage separately.

\textbf{Stage I:} Since $m$ is sufficiently large, $k \leq (m - 25)/2$ and $|V(G) \setminus \{x\}| = m-1$
is even, it follows by Lemma~\ref{lem::MissingEdgesPM} that Maker can follow the proposed strategy for this stage.

\textbf{Stage II:} It follows by its description that Maker can follow the proposed strategy for this stage.

\textbf{Stage III:} In order to prove that Maker can follow the proposed strategy for this stage without
forfeiting the game, we will first prove the following three claims.

\begin{claim} \label{cl::lengthOfStageII}
Maker plays at most $(m + 2k + 3)/4$ moves in Stage II.
\end{claim} 

\begin{proof}
Since Breaker claims exactly $(m-1)/2 + 2$ edges of $G$ before Maker's first move in 
Stage II, it follows that $e(Br) \leq (m-1)/2 + 2 + k$ holds at that point.
In each round (that is, a move of Maker and a counter move of Breaker) of Stage II, 
$e(Br)$ is decreased by at least $2$ (it is decreased by $D(e) \geq 3$ in Maker's 
move and then increased by at most $1$ in Breaker's move). The claim now follows since
$e(Br) \geq 0$ holds at any point during the game. 
\end{proof}

\begin{claim} \label{cl::freeEdgesStageII}
$e(Br) \leq |End| - 3$ holds at any point during Stage II.
\end{claim}

\begin{proof}
At the end of Stage I, Maker's graph consists of $(m-5)/2$ paths of length 1 each, 
1 path of length 3, and 1 special path $P_0 = \{x\}$ of length 0. Hence, $|End| = m - 2$
holds at the beginning of Stage II. Since Breaker claims exactly $(m-1)/2 + 2$ edges of 
$G$ before Maker's first move of Stage II, it follows that $e(Br) \leq (m-1)/2 + 2 + k \leq 
m - 5 = |End| - 3$ holds at that point, where the last inequality holds by the assumed upper bound 
on $k$. Assume that $e(Br) \leq |End| - 3$ holds immediately after Breaker's $j$th move for
some integer $j \geq (m-1)/2 + 2$. If Maker plays her $j$th move in Stage II, then she claims 
an available edge $e$ such that $D(e) \geq 3$. This decreases $|End|$ by $2$ and $e(Br)$ 
by at least $3$. It follows that $e(Br) \leq |End| - 4$ holds immediately after Maker's 
$j$th move. In his $(j+1)$st move, Breaker increases $e(Br)$ by at most $1$ and does not 
decrease $|End|$. Hence $e(Br) \leq |End| - 3$ holds immediately after his $(j+1)$st move.  
\end{proof}

\begin{claim} \label{cl::endOfStageII}
The following three properties hold immediately before Maker's first move of Stage III:
\begin{description}
\item [(i)] $|End| \geq (m - 2k - 7)/2$. 
\item [(ii)] $\Delta(Br) \leq 2$.
\item [(iii)] $Br$ is a matching or a subgraph of $K_3$ or a subgraph of $C_4$ whose vertices 
are $End(P_i) \cup End(P_j)$ for some $1 \leq i < j \leq \ell$.
\end{description} 
\end{claim} 

\begin{proof}
As shown in the proof of Claim~\ref{cl::freeEdgesStageII}, $|End| = m - 2$
holds at the beginning of Stage II. In each of her moves in Stage II, Maker decreases
$|End|$ by exactly 2. Since, by Claim~\ref{cl::lengthOfStageII} Maker plays at most
$(m + 2k + 3)/4$ moves in Stage II, it follows that $|End| \geq (m-2) - (m + 2k + 3)/2 
= (m - 2k - 7)/2$ holds at the end of Stage II; this proves (i).  

Next, we prove (ii). suppose for a contradiction that there are vertices $u, v_1, v_2, v_3 \in End$
such that $u v_1, u v_2, u v_3 \in E(Br)$ at the end of Stage II. It follows by Claim~\ref{cl::freeEdgesStageII} 
that there exists a vertex $v_4 \in End$ such that the edge $u v_4$ is available. Clearly 
$u v_1, u v_2, u v_3 \in \left\{f \in E(Br) : u v_4 \cap f \neq \emptyset \right\}$. Therefore, 
$D(u v_4) \geq 3$ contrary to our assumption that Stage II is over.  

Finally, we prove (iii). It follows by (ii) that $\Delta(Br) \leq 2$. If $\Delta(Br) \leq 1$,
then $Br$ is a matching. Assume then that there are vertices $u, v, w \in End$ such that 
$uv, uw \in E(Br)$. Let $1 \leq i \leq \ell$ be the unique index such that $u \in V(P_i)$ and
let $u' = End(P_i) \setminus \{u\}$. We claim that $d_{Br}(z) = 0$ for every $z \in End \setminus 
\{u, v, w, u'\}$. Indeed, suppose for a contradiction that there exist vertices 
$z \in End \setminus \{u,v,w,u'\}$ and $z' \in End$ such that $zz' \in E(Br)$. Since 
$\Delta(Br) \leq 2$, $z \notin \{u,v,w,u'\}$ and $uv, uw \in E(Br)$, it follows that $uz$ is available. 
However, we then have $uv, uw, zz' \in \{f \in E(Br) : u z \cap f \neq \emptyset\}$. Therefore, 
$D(uz) \geq 3$ contrary to our assumption that Stage II is over. If $d_{Br}(u') = 0$ as well,
then $E(Br) \subseteq \{uv,uw,vw\}$, that is, $Br$ is a subgraph of $K_3$. Assume then without 
loss of generality that $u'w \in E(Br)$. Since $\Delta(Br) \leq 2$ holds by (ii), it follows that
$vw \notin E(Br)$. If on the other hand $vw$ is available, then $uv, uw, u'w \in \{f \in E(Br) : 
vw \cap f \neq \emptyset\}$ contrary to our assumption that Stage II is over. It follows that    
$\{v,w\} = End(P_j)$ for some $1 \leq j \leq \ell$ and that $E(Br) \subseteq \{uv, uw, u'v, u'w\}$.
\end{proof}

We can now prove that Maker can follow the proposed strategy for this stage without
forfeiting the game. While doing so we will also show that she plays at most $2$ moves 
in Stage III. It follows by Part (iii) of Claim~\ref{cl::endOfStageII} that, immediately before Maker's first move 
in Stage III, the graph $Br$ is a matching or a subgraph of $K_3$ or a subgraph of $C_4$ whose vertices 
are $End(P_i) \cup End(P_j)$ for some $1 \leq i < j \leq \ell$. In the first case, $\Delta(Br) \leq 1$ and thus 
Maker plays no moves in Stage III. Next, assume that $\{uv, uw\} \subseteq E(Br) \subseteq \{uv, uw, vw\}$ for some 
$u,v,w \in End$. Assume without loss of generality that Maker claims $uy$ in her first move of Stage III. Since 
$e(Br) \leq 3$ holds immediately before this move, it follows by Part (i) of Claim~\ref{cl::endOfStageII} and by
the assumed upper bound on $k$ from Lemma~\ref{lem::PathOneEndpoint} that such an available edge exists. Let $z z'$ 
denote the edge claimed by Breaker in his subsequent move. Note that $E(Br) \subseteq \{vw, z z'\}$ holds at this point. 
If $\{v,w\} \cap \{z,z'\} = \emptyset$, then $Br$ is a matching and Stage III is over. Assume then without loss of 
generality that $v = z$. In her second move of Stage III, Maker claims an available edge $v z''$. Since $e(Br) \leq 2$ 
holds immediately before this move, it follows that such an available edge exists. Clearly, $e(Br) \leq 1$ must hold 
after Breaker's next move. It follows that Maker will not play any additional moves in Stage III. Finally, assume that 
there are indices $1 \leq i < j \leq \ell$ such that $End(P_i) = \{u,u'\}$, $End(P_j) = \{v,v'\}$ and $E(Br) \subseteq 
\{uv, uv', u'v, u'v'\}$. Assume without loss of generality that Maker claims $uy$ in her first move of Stage III. 
Since $e(Br) \leq 3$ holds immediately before this move, it follows that such an available edge exists. Let $z z'$ 
denote the edge claimed by Breaker in his subsequent move. Note that $E(Br) \subseteq \{u'v, u'v', z z'\}$ holds 
at this point. Since $vv' \notin X$, it follows that $zz' \neq vv'$; assume without loss of generality that 
$z \notin \{v,v'\}$. In her second move of Stage III, Maker claims $u' z$ if $z' \neq u'$ and an available edge 
$u' z''$ otherwise. Since $e(Br) \leq 3$ holds immediately before this move, it follows that such an available 
edge exists. Clearly, $e(Br) \leq 1$ must hold after Breaker's next move. It follows that Maker will not play any 
additional moves in Stage III. 

\textbf{Stage IV:} In order to prove that Maker can follow the proposed strategy for this stage without
forfeiting the game, we will first prove the following two claims.

\begin{claim} \label{cl::manyPaths}
At the end of Stage III, Maker's graph consists of at least 4 paths.
\end{claim}

\begin{proof}
It follows by Part (i) of Claim~\ref{cl::endOfStageII} that $|End| \geq (m - 2k - 7)/2$
holds at the end of Stage II. Since, as noted above, Maker plays at most 2 moves in Stage III,
it follows that $|End| \geq (m - 2k - 11)/2 \geq 7$ holds at the end of that stage, where the
last inequality holds by the assumed upper bound on $k$. The claim readily follows.
\end{proof}

\begin{claim} \label{cl::StageIV}
The following two properties hold immediately after each of Maker's moves in this stage:
\begin{description}
\item [(i)] $d_{Br}(p_0) = 0$. 
\item [(ii)] $\Delta(Br) \leq 1$.
\end{description}  
\end{claim}

\begin{proof}
It follows by the description of Stage III of the proposed strategy that Property (ii) holds 
immediately before Maker's first move in Stage IV. It thus follows by the description of Maker's 
first move in this stage, that both properties hold after this move. Assume then that both properties
hold immediately after Maker's $i$th move of this stage for some $i \geq 1$. Let $uv$ denote 
the edge claimed by Breaker in his $i$th move of this stage (recall that Maker is the first 
to play in Stage IV), where $u \neq p_0$. Assume that $uv \in X$ as otherwise there is nothing 
to prove. Note that $d_{Br}(w) \leq 1$ holds for every $w \in End \setminus \{u,v\}$ at this point. 
Unless she forfeits the game, in her $(i+1)$st move of this stage, Maker claims an available edge
$uw$ such that $w \in End \setminus \{p_0\}$. This does not change $p_0$, removes $u$ from $End$
and decreases $d_{Br}(v)$ by 1. It follows that $d_{Br}(v) \leq 1$ and that $d_{Br}(v) = 0$ if 
$v = p_0$. 
\end{proof}

It follows by Claim~\ref{cl::manyPaths} and by the description of the proposed strategy for Stage IV that 
$|End| \geq 7$ holds immediately before each of Maker's moves in Stage IV. It thus follows by Property (ii) 
from Claim~\ref{cl::StageIV} that Maker can follow the proposed strategy for this stage without forfeiting 
the game.

\textbf{Stage V:} It follows by Claim~\ref{cl::manyPaths} and by the description of the proposed strategy for 
Stage IV that Maker's graph consists of exactly 3 paths (one of which is $p_0$) in the beginning of Stage V.
Using Properties (i) and (ii) from Claim~\ref{cl::StageIV}, one can show via a simple case analysis (whose 
details we omit) that, regardless of Breaker's strategy, Maker can claim two available edges such that the 
resulting graph is a Hamilton path with $x$ as an endpoint.   
\end{proof}

We now turn to the proof of Theorem~\ref{th::randomTree} whose main idea is the following.
Similarly to the proof of Theorem~\ref{th::manyLeaves} given in Section~\ref{sec:manyleaves}, 
Maker starts by embedding a tree $T'' \subseteq T$ while limiting Breaker's degrees in certain 
vertices. In contrast to the proof of Theorem~\ref{th::manyLeaves}, where $T \setminus T''$ is 
a matching of linear size, in the current proof $T \setminus T''$ consists of linearly 
many pairwise vertex-disjoint bare paths of length $k$ each, where $k$ is a fixed large constant.
We then embed the paths of $T \setminus T''$, recalling that for each of them, one endpoint
was previously embedded. The main tool used for this latter part is Lemma~\ref{lem::PathOneEndpoint}.

In order to prove Theorem~\ref{th::randomTree} we will require the following results.

\begin{theorem} [Theorem 3 in~\cite{Moon}] \label{th::maxDeg}
Let $T$ be a tree, chosen uniformly at random from the class of all labeled trees on $n$ vertices. Then 
asymptotically almost surely, $\Delta(T) = (1 + o(1)) \log n/ \log \log n$.
\end{theorem}

\begin{lemma} \label{lem::ManyBarePaths}
For every positive integer $k$ there exists a real number $\varepsilon > 0$ such that the following holds for 
every sufficiently large integer $n$. Let $T$ be a tree, chosen uniformly at random from the class of all 
labeled trees on $n$ vertices. Then asymptotically almost surely $T$ is such that there exists a family 
${\mathcal P}$ which satisfies all of the following properties:
\begin{description}
\item [(1)] Every $P \in {\mathcal P}$ is a bare path of length $k$ in $T$.
\item [(2)] $|{\mathcal P}| \geq \varepsilon n$. 
\item [(3)] For every $P \in {\mathcal P}$, one of the vertices in $End(P)$ is a leaf of $T$.
\item [(4)] If $P_1 \in {\mathcal P}$ and $P_2 \in {\mathcal P}$ are two distinct paths, then $V(P_1) \cap V(P_2) = \emptyset$. 
\end{description}
\end{lemma}

Lemma~\ref{lem::ManyBarePaths} is an immediate corollary of Lemma 3 from~\cite{AHK}; we omit the straightforward details.

\begin{lemma} \label{lem::boardPartition}
Let $k$ and $q$ be integers and let $X$ and $Y$ be sets such that $|X| = q$ and $|Y| = kq$. Let 
$H$ be a graph, where $V(H) = X \cup Y$, which satisfies the following properties:  
\begin{description}
\item [(a)] $\Delta(H[Y]) \leq q-1$.
\item [(b)] $d_H(u, Y) \leq q/2$ for every $u \in X$.
\item [(c)] $d_H(u, X) \leq q/(2k)$ for every $u \in Y$.
\end{description} 
Then there exists a partition $V(H) = V_1 \cup \ldots \cup V_q$ 
such that the following properties hold for every $1 \leq i \leq q$:
\begin{description}
\item [(i)] $|X \cap V_i| = 1$. 
\item [(ii)] $|Y \cap V_i| = k$.
\item [(iii)] $E(H[V_i]) = \emptyset$.
\end{description} 
\end{lemma}

In the proof of Lemma~\ref{lem::boardPartition} we will make use of the following well known 
result due to Hajnal and Szemer\'edi~\cite{HSz}.

\begin{theorem} [Theorem 1 in~\cite{HSz}] \label{th::HajnalSzemeredi}
Let $G$ be a graph on $n$ vertices and let $r$ be a positive integer. If $\Delta(G) \leq r-1$, then
there exists a proper $r$-colouring of the vertices of $G$ such that every colour class has size
$\lfloor n/r \rfloor$ or $\lceil n/r \rceil$. 
\end{theorem}

\textbf{Proof of Lemma~\ref{lem::boardPartition}}
Since $\Delta(H[Y]) \leq q-1$ holds by Property (a), it follows by Theorem~\ref{th::HajnalSzemeredi} 
that there exists a partition $Y = U_1 \cup \ldots \cup U_q$ such that $|U_i| = k$ and $E(H[U_i]) = \emptyset$ 
hold for every $1 \leq i \leq q$. Let $U = \{U_1, \ldots, U_q\}$ and let $G$ be the bipartite graph with 
parts $X := \{x_1, \ldots, x_q\}$ and $U$ where, for every $1 \leq i, j \leq q$ there is an edge of $G$ between 
$x_i$ and $U_j$ if and only if $d_G(x_i, U_j) = 0$. Since $\delta(G) \geq q/2$ holds by Properties (b) and (c), 
it follows by Hall's Theorem (see, e.g.~\cite{West}) that $G$ admits a perfect matching. Assume without loss
of generality that $\{x_i U_i : 1 \leq i \leq q\}$ is such a matching. For every $1 \leq i \leq q$ let 
$V_i = U_i \cup \{x_i\}$. It is easy to see that the partition $V(H) = V_1 \cup \ldots \cup V_q$ satisfies   
Properties (i), (ii) and (iii).
{\hfill $\Box$ \medskip\\}

\textbf{Proof of Theorem~\ref{th::randomTree}}
Let $k$ be a sufficiently large integer (e.g. $m_0$ from Lemma~\ref{lem::PathOneEndpoint} is large enough) and 
let $n$ be sufficiently large with respect to $k$. Let $T$ be a tree, chosen uniformly at random from the class 
of all labeled trees on $n$ vertices. It follows by Theorem~\ref{th::maxDeg} that, asymptotically almost surely, 
$\Delta(T) = (1 + o(1)) \log n/ \log \log n$ and by Lemma~\ref{lem::ManyBarePaths} that there exists a family 
${\mathcal P}$ of $\varepsilon n$ pairwise vertex-disjoint bare paths of $T$, such that for every $P \in {\mathcal P}$, 
$P = (v^P_0 \ldots v^P_k)$ and $v^P_k$ is a leaf of $T$. From now on we will thus assume that the tree $T$ satisfies
these properties. 

Let
$$
T' = T \setminus \left(\bigcup_{P \in {\mathcal P}} \left(V(P) \setminus \{v^P_0\} \right) \right) \,.
$$
Throughout the game, Maker maintains a set $S \subseteq V(T)$ of embedded vertices, an $S$-partial embedding 
$f$ of $T$ in $K_n \setminus B$ and a set $A = V(K_n) \setminus f(S)$ of available vertices. Initially, 
$S = \{v'\}$ and $f(v') = v$, where $v' \in V(T')$ and $v \in V(K_n)$ are arbitrary vertices.

First we describe a strategy for Maker in $(E(K_n), {\mathcal T}_n)$ and then prove that it allows her to 
build a copy of $T$ within $n-1$ moves. At any point during the game, if Maker is unable to follow the 
proposed strategy, then she forfeits the game. Certain parts of the proposed strategy are very similar to 
the strategy described in the proof of Theorem~\ref{th::manyLeaves}. Therefore, we describe these parts
rather briefly while elaborating considerably where the two strategies differ. The proposed strategy is 
divided into the following three stages.

\textbf{Stage I:} Maker builds a tree $T''$ such that the following properties hold at the 
end of this stage:  
\begin{description}
\item [(1)] $T' \subseteq T'' \subseteq T$.  
\item [(2)] $d_B(v) \leq 2\sqrt{n} \log n$ for every vertex $v \in A \cup f({\mathcal O}_T)$.
\item [(3)] $|\{P \in {\mathcal P} : v^P_1 \in S\}| \leq \sqrt{n}$ (in particular, $|V(T'')| \leq n - \varepsilon n$).
\end{description}
Moreover, Maker does so in exactly $|V(T'')| - 1$ moves.

\textbf{Stage II:} In this stage Maker completes the embedding of every path $P \in {\mathcal P}$ which was 
partially embedded in Stage I. For every $P \in {\mathcal P}$, let $0 \leq i_P \leq k$ denote the largest 
integer such that $v^P_{i_P} \in S$. For as long as there exists a path $P \in {\mathcal P}$ for which
$0 < i_P < k$, Maker plays as follows. She picks an arbitrary path $P \in {\mathcal P}$ for which $0 < i_P < k$ 
and claims an arbitrary free edge $f(v^P_{i_P}) u$, where $u \in A$. Subsequently, Maker updates $S$ and $f$
by adding $v^P_{i_P+1}$ to $S$ and setting $f(v^P_{i_P+1}) = u$.

As soon as $i_P \in \{0, k\}$ holds for every $P \in {\mathcal P}$, Stage II is over and
Maker proceeds to Stage III.

\textbf{Stage III:} Let $f({\mathcal O}_T) = \{x_1, \ldots, x_q\}$ and let $A \cup 
\{x_1, \ldots, x_q\} = V_1 \cup \ldots \cup V_q$ be a partition of $A \cup \{x_1, \ldots, x_q\}$ 
such that the following properties hold for every $1 \leq i \leq q$:
\begin{description}
\item [(a)] $|V_i| = k+1$. 
\item [(b)] $x_i \in V_i$.
\item [(c)] $E(B[V_i]) = \emptyset$.
\end{description}
For every $1 \leq i \leq q$ let ${\mathcal S}_i$ be a strategy for building a Hamilton path of 
$(K_n \setminus B)[V_i]$ such that $x_i$ is one of its endpoints in $|V_i| - 1$ moves. 
Maker plays $q$ such games in parallel, that is, whenever Breaker claims an edge of $K_n[V_i]$
for some $1 \leq i \leq q$ for which $M[V_i]$ is not yet a Hamilton path, Maker plays in 
$(K_n \setminus B)[V_i]$ according to ${\mathcal S}_i$. In all other cases, she plays in
$(K_n \setminus B)[V_j]$ according to ${\mathcal S}_j$, where $1 \leq j \leq q$ is an arbitrary
index for which $M[V_j]$ is not yet a Hamilton path.

It is evident that if Maker can follow the proposed strategy without
forfeiting the game, then she builds a copy of $T$ in $n-1$ moves. It
thus suffices to prove that Maker can indeed do so. We consider each
of the three stages separately.

\textbf{Stage I:} The exact details of Maker's strategy for this stage and the
proof that she can follow it without forfeiting the game are essentially the same as 
those for Stage I in the proof of Theorem~\ref{th::manyLeaves}. There are a few 
differences which arise since $\Delta(T)$ is not bounded (but not too large 
either -- see Theorem~\ref{th::maxDeg}) and since $T \setminus T'$ consists of 
pairwise vertex-disjoint long bare paths, rather than a matching. Defining a vertex
$v \in A \cup f({\mathcal O}_T)$ to be dangerous if $d_B(v) \geq \sqrt{n} \log n$ 
ensures that at most $2 \sqrt{n}/\log n$ vertices become dangerous throughout Stage I 
similarly to Claim~\ref{claim1}. Since the paths in ${\mathcal P}$ are pairwise vertex-disjoint, 
$\Delta(T) = o(\log n)$ and $2\sqrt{n} \log n \leq \varepsilon n/(10 \Delta(T))$, it follows that  
Claims~\ref{claim2} and~\ref{claim3} hold as well. The remaining details are omitted.  

\textbf{Stage II:} Since $e(P) = k$ holds for every $P \in {\mathcal P}$, it follows
by Property (3) that Stage II lasts $O(k \sqrt{n})$ moves and that $|A| = \Theta(n)$ 
holds at any point during this stage. Since $n$ is sufficiently large with respect to $k$, 
it follows by Property (2) that $d_B(v) = O(\sqrt{n} \log n)$ holds for every vertex 
$v \in A \cup f({\mathcal O}_T)$ at any point during this stage. We conclude that Maker 
can indeed follow the proposed strategy for this stage. 

\textbf{Stage III:} Since, as noted above, $d_B(v) = O(\sqrt{n} \log n)$ holds for every vertex $v \in A \cup 
f({\mathcal O}_T)$ at the end of Stage II and since $n$ is sufficiently large with respect to $k$, it follows 
by Lemma~\ref{lem::boardPartition} that the required partition exists. Moreover, it follows by Property (c), 
by the choice of $k$ and by Lemma~\ref{lem::PathOneEndpoint} that Maker can follow the proposed strategy 
for this stage.
{\hfill $\Box$ \medskip\\}

We end this section with a proof of Theorem~\ref{th::PathFromLeaf}. The main idea is similar to the proof of
Theorem~\ref{th::longBarePath} given in Section~\ref{sec::TwoTrees}. That is, we first embed 
the tree $T$ except for a sufficiently long bare path $P$ between a leaf and another vertex and then
embed $P$, recalling that one of its endpoints was already embedded. We will do so without wasting any moves.
We can thus use Theorem~\ref{TreesWithLongPath} for the former and Lemma~\ref{lem::PathOneEndpoint} for the latter.  

\textbf{Proof of Theorem~\ref{th::PathFromLeaf}}
Let $k = \binom{\D}{2} + 1$, let $m_0 = m_0(k)$ be the constant whose existence follows
from Lemma~\ref{lem::PathOneEndpoint} and let $m_3 = \max \{m_0, (\D+1)^2\}$. Let $P$ be a bare path in $T$ 
of length $m_3$ with endpoints $x'_1$ and $x'_2$, where $x'_2$ is a leaf. Let $T'$ be the tree which is obtained from $T$ 
by deleting all the vertices in $V(P) \setminus \{x'_1\}$. 

Maker's strategy consists of two stages. In the first stage she embeds $T'$ using the strategy
whose existence follows from Theorem~\ref{TreesWithLongPath} (with $r=1$) while ensuring that Properties 
(i) and (ii) are satisfied. Let $f : T' \rightarrow M$ be an isomorphism, let $x_1 = f(x'_1)$, 
let $A = V(K_n) \setminus f(V(T'))$, let $U = A \cup \{x_1\}$ and let $G = (K_n \setminus B)[U]$.

In the second stage she embeds $P$ into $G$ such that $x_1$ is the non-leaf endpoint. She does so 
using the strategy whose existence follows from Lemma~\ref{lem::PathOneEndpoint} which is applicable by the choice 
of $m_3$ and by Property (ii). Hence, $T \subseteq M$ holds at the end of the second stage, that is, 
Maker wins the game.

It follows by Theorem~\ref{TreesWithLongPath} that the first stage lasts exactly 
$v(T') - 1 = n - |V(P)| = n - |U|$ moves. It follows by Lemma~\ref{lem::PathOneEndpoint} that the second 
stage lasts exactly $|U| - 1$ moves. Therefore, the entire game lasts exactly $n-1$ moves as claimed.
{\hfill $\Box$ \medskip\\}

\section{Concluding remarks and open problems}
\label{sec::openprob}

\textbf{Building trees in the shortest possible time}. 

As noted in the introduction, there are trees $T$ on $n$
vertices with bounded maximum degree which Maker cannot build in $n-1$ moves. In this paper we proved that Maker 
can build such a tree $T$ in at most $n$ moves if it admits a long bare path and in at most $n+1$ moves if it 
does not. We do not believe that there are bounded degree trees that require Maker to waste more than one move. 
This leads us to make the following conjecture.

\begin{conjecture} \label{conj::nMoves}
Let $\Delta$ be a positive integer. Then there exists an integer
$n_0 = n_0(\Delta)$ such that for every $n \geq n_0$ and for every
tree $T=(V,E)$ with $|V|=n$ and $\Delta(T) \leq \Delta$, Maker has a strategy 
to win the game $(E(K_n), {\mathcal T}_n)$ within $n$ moves.
\end{conjecture} 

It follows by Theorem~\ref{th::longBarePath} that the assertion of Conjectute~\ref{conj::nMoves} is true
for bounded degree trees which admit a long bare path; the problem is with trees that do not admit such
a path. Nevertheless, we can prove Conjectute~\ref{conj::nMoves} for many (but not all) such trees as well.
For example, we can prove (but omit the details) that Maker has a strategy to build a complete binary tree 
in $n$ moves (recall from the introduction that this is tight).   

\textbf{Building trees without wasting moves}. 

As previously noted, there are trees which Maker can build in 
$n-1$ moves (such as the path on $n$ vertices) and there are trees which require at least $n$ moves (such as
the complete binary tree). It would be interesting to characterize the family of all (bounded degree) trees
on $n$ vertices which, playing on $K_n$, Maker can build in exactly $n-1$ moves. 

\textbf{Strong tree embedding games}. 

As noted in~\cite{FH2011}, an explicit very fast winning strategy for Maker in a weak game can sometimes
be adapted to an explicit winning strategy for Red in the corresponding strong game. Since it was proved 
in~\cite{FHK2012} that Maker has a strategy to win the weak tree embedding game $(E(K_n), {\mathcal T}_n)$ 
within $n + o(n)$ moves, it was noted in~\cite{FH} that one could be hopeful about the possibility of 
devising an explicit winning strategy for Red in the corresponding strong game. The first step towards 
this goal is to find a much faster strategy for Maker in the weak game $(E(K_n), {\mathcal T}_n)$. This 
was accomplished in the current paper.    

\textbf{Building trees quickly on random graphs}.

The study of fast winning strategies for Maker on random graphs was initiated in~\cite{CFKL2012}. The problem of 
determining the values of $p = p(n)$ for which asymptotically almost surely Maker can win $(E(G(n,p)), {\mathcal T}_n)$ 
quickly (say, within $n + o(n)$ moves), where $T$ is any tree with bounded maximum degree was raised in that paper. 
Note that the game $(E(K_n), {\mathcal T}_n)$ studied in this paper is the special case with $p=1$. It seems plausible 
that the methods developed in the current paper combined with those of~\cite{CFKL2012} could be helpful when addressing this problem. 
It should however be noted that the exact threshold for the appearance in $G(n,p)$ of some fixed bounded degree spanning tree is not 
known in general and is an important open problem in the theory of random graphs. It would thus be very hard to answer the 
analogous general question for games. However, the exact threshold is known in several special cases, such as trees with linearly 
many leaves~\cite{HKS}. Therefore, one could try to adjust the proof of Theorems~\ref{th::manyLeaves} to $G(n,p)$ with $p$ being as close as possible 
to the threshold $\log n/n$. Moreover, a weaker (but far from trivial) general upper bound was proved in~\cite{Kri} so one could 
at least try to prove that Maker wins quickly on these denser random graphs. Finally, note that some results about $(E(G(n,p)), {\mathcal T}_n)$,
where $p \geq C n^{-1/3} \log^2 n$, follow from more general results proved in~\cite{JKS}. However, even for this range of probabilities, a different 
argument is needed in order to prove that Maker wins quickly.

\section*{Acknowledgement}

We would like to thank Michael Krivelevich for helpful comments.

\end{document}